\documentclass{amsart}[2000]

\usepackage{amsmath}
\usepackage{amssymb}
\usepackage{amsthm}
\usepackage[shortlabels]{enumitem}
\usepackage{mathtools}
\usepackage{color}
\usepackage{bbm}
\usepackage{tikz-cd}
\usepackage{soul}
\usepackage[all]{xy}
\usepackage{hyperref}

\theoremstyle{plain}
\newtheorem{theorem}{Theorem}[section]
\newtheorem{corollary}[theorem]{Corollary}
\newtheorem{lemma}[theorem]{Lemma}

\newtheorem{proposition}[theorem]{Proposition}

\theoremstyle{definition}
\newtheorem{definition}[theorem]{Definition}

\newtheorem{example}[theorem]{Example}
\newtheorem{remark}[theorem]{Remark}

\numberwithin{equation}{section}

\newcommand{\A}{{\mathcal{A}}}

\renewcommand{\H}{{\mathcal{H}}}
\newcommand{\I}{{\mathcal{I}}}
\newcommand{\J}{{\mathcal{J}}}
\newcommand{\K}{{\mathcal{K}}}

\newcommand{\R}{{\mathcal{R}}}
\renewcommand{\S}{{\mathcal{S}}}
\newcommand{\T}{{\mathcal{T}}}

\newcommand{\B}{{\mathbb{B}}}
\newcommand{\F}{{\mathbb{F}}}
\renewcommand{\K}{{\mathbb{K}}}
\newcommand{\N}{{\mathbb{N}}}
\renewcommand{\R}{{\mathbb{R}}}
\newcommand{\Z}{{\mathbb{Z}}}

\newcommand{\Alg}{\operatorname{Alg}}

\newcommand{\Gr}{\operatorname{Gr}}

\newcommand{\id}{{\operatorname{id}}}

\renewcommand{\sp}{\operatorname{sp}}
\DeclareMathOperator{\Ad}{Ad}
\DeclareMathOperator{\supp}{supp}

\DeclareMathOperator\ST{ST}
\DeclareMathOperator\RL{RL}
\DeclareMathOperator\env{env}
\DeclareMathOperator{\SOT}{SOT}

\usepackage{etoolbox}
\apptocmd{\sloppy}{\hbadness 10000\relax}{}{}
\apptocmd{\sloppy}{\vbadness 10000\relax}{}{}

\title[Space-time boundaries and operator algebras for random walks]{Space-time boundaries for random walks \\ and their application to operator algebras}
\subjclass[2020]{Primary: 60J50, 60G50. Secondary: 60J10, 47L55.}
\keywords{random walks, space-time boundary, harmonic functions, operator algebras, boundary representations}

\author[A. Dor-On]{Adam Dor-On}
\address{Department of Mathematics \\ University of Haifa \\ Haifa \\ Israel.}
\email{adoron.math@gmail.com}

\author[I. Gekhtman]{Ilya Gekhtman}
\address{Department of Mathematics \\ Technion - IIT \\ Haifa \\ Israel.}
\email{ilyagekh@gmail.com}

\author[P. Prudnikov]{Pavel Prudnikov}
\address{Department of Mathematics \\ Technion - IIT \\ Haifa \\ Israel.}
\email{psprudnikov@icloud.com}

\thanks{A. Dor-On was partially supported by an NSF-BSF grant no. 2350543 / 2023695 (respectively) and an ISF grant no. 2499/25. I. Gekhtman was partially supported by an ISF grant no. 3423/24.}

\begin{document}

\begin{abstract}
    
    We investigate the Martin boundary of the space-time Markov chain associated to a finitely supported random walk $(\Gamma, \mu)$ with spectral radius $\rho$ and relate it to several classical compactifications of $\Gamma$. Assuming the strong ratio-limit property, we prove that the reduced ratio-limit compactification embeds naturally into the space-time Martin boundary. We introduce the $0$-Martin boundary, which governs the behaviour of $\infty$-harmonic functions, and show that the $0$-Martin kernels arise as rescaled limits of $\lambda$-Martin kernels as $\lambda\rightarrow 0$. For symmetric random walks on hyperbolic groups, the $0$-Martin boundary naturally covers the Gromov boundary, while the cover need not be injective in general. Our main structural theorem identifies the minimal space-time Martin boundary with the disjoint union of minimal $\lambda$-Martin boundaries over $\lambda\in [0, \rho^{-1}]$ with its natural pointwise topology. As an application, we show that the noncommutative Shilov boundary of the tensor algebra of the random walk $(\Gamma, \mu)$ coincides with its Toeplitz $C^*$-algebra.
\end{abstract}

\maketitle

\section{Introduction}
\label{s:introduction}

A popular line of research in geometric group theory and probability studies various boundaries of groups associated to random walks on them. Measureable and topological approaches to constructing such boundaries were initiated by Furstenberg \cite{Fur71, Fur72}, going back to ideas of Poisson in classical harmonic analysis. One such boundary, which has been studied for over fifty years is the so called Martin boundary, going back to ideas from classical potential theory \cite{Mar41}.

Let $\Gamma$ be a discrete countable group, and $\mu\colon\Gamma \rightarrow [0,1]$ be a probability vector whose support generates $\Gamma$ as a semigroup (sometimes called admissible). The \emph{random walk} on $\Gamma$ induced by $\mu$ is the stochastic matrix $P$ of transition probabilities given by $P(x,y):= \mu(x^{-1}y)$ for $x,y \in \Gamma$. To define the Martin boundary, we denote by $P^{n}(x, y) := \mu^{*n}(x^{-1}y)$ the $n$-step transition probabilities, and by $G(x, y) := \sum^\infty_{n = 0} P^{n}(x, y)$ the associated Green function. The Martin boundary $\partial_{M}\Gamma$ of the random walk $(\Gamma, \mu)$ is defined as the boundary obtained from compactifying $\Gamma$ with respect to the Martin kernel functions $\{y\mapsto G(x, y)/G(e, y)\}_{x\in \Gamma}$. We will denote the spectral radius of $\mu$ by $\rho := \lim \sup_{n\to \infty}(\mu^{\ast n}(e))^{1/n}$. For any $\lambda\in (0, \rho^{-1}]$, we can similarly define the $\lambda$-Green function $G(x, y \, | \, \lambda) := \sum^\infty_{n = 0}\lambda^{n}P^{n}(x, y)$ and the $\lambda$-Martin kernel functions, as well as the $\lambda$-Martin boundary $\partial_{M, \lambda}\Gamma$ of the random walk for $(\Gamma, \mu)$ (See Section \ref{s:preliminaries}). When $\mu$ is finitely supported, by the Martin--Poisson integral representation theorem, we know that $\lambda$-harmonic functions which are normalized to be $1$ at the identity element of $\Gamma$ correspond to probability measures supported on the $\lambda$-Martin boundary. The set of points in the $\lambda$-Martin boundary which correspond to extreme points of the compact convex set of normalized $\lambda$-harmonic functions (whose elements are often referred to as minimal $\lambda$-harmonic functions) is called the minimal $\lambda$-Martin boundary.

Martin boundaries have been identified for large classes of random walks when the group admits some kind of intrinsic geometric boundary. For instance, for symmetric finitely supported random walks on hyperbolic groups, for $\lambda \in (0,\rho^{-1}]$ the $\lambda$-Martin boundary has been shown to coincide with Gromov boundary \cite{Ancona, Lalley-Gouezel13, GouezelLLT}, and for symmetric finitely supported random walks on relatively hyperbolic groups it was shown that $\lambda$-Martin boundaries cover the Bowditch boundary of the group \cite{GGPY}, with a precise description of the fibers of the cover when parabolic subgroups are virtually abelian \cite{DGGP}.
 
Recently, the first-named author has established connections between topological boundary theory for Markov chains and analogues of classical Toeplitz C*-algebras \cite{DOAD21}. An important player that emerged from this study is a compactification associated to the random walk called the \emph{ratio-limit boundary}. This boundary was further studied and described for various examples, for instance by Woess \cite{WOWO21} for random walks on hyperbolic groups and in \cite{DDG+} for random walks on relatively hyperbolic groups. 

In this paper, we continue to explore the connection between operator algebras and random walks on groups, motivated by problems from the theory of operator algebras. In particular, we study a compactification of the space-time graph associated to a random walk, called the \emph{space-time Martin compactification}, which heuristically encapsulates both the ratio limit boundary and the $\lambda$-Martin boundaries, as well as a certain limiting case of $\lambda$-Martin boundaries as $\lambda \rightarrow 0$ that we call the $0$-Martin boundary. Our work was motivated by the problem of characterizing the non-commutative Shilov boundary of the tensor algebra associated to a random walk as studied in \cite{DOMA14, DOMA17}, and we use our description of the minimal space-time boundary to resolve this problem.

We start with discussing the space-time Martin boundary of a random walk. The space-time set $\ST$ is defined as the subset of $\Gamma\times\mathbb{N}$ given by pairs $(x,m)$ such that $P^m(e,x)>0$, and an edge from $(x, m)$ to $(y, n)$ if and only if $n = m + 1$ and $\mu(x^{-1}y)>0$. We then define a Markov chain $(Y_n)_{n\in \N}$ on $\ST$ with transition probabilities $P_{\ST}((x, m), (y, m + 1)) = \mu(x^{-1}y)$. We call it the space-time Markov chain associated to $\mu$. The Green function and Martin boundary can be similarly defined for the space-time Markov chain, and the space-time Martin boundary $\partial_{\ST}\Gamma$ of $(\Gamma, \mu)$ is defined to be the Martin boundary of the \emph{space-time} Markov chain. The space-time Markov chain and its boundary have been studied in \cite{LS63, McDM21, Lalleyspacetime}. However, to the best of our knowledge there has not been much work on relating the space-time boundary to other standard boundaries of Markov chains, or computing it for random walks on specific groups.

For what follows, we will need to recall the ratio limit boundary. Given a random walk on a countable discrete group $\Gamma$ with an admissible measure $\mu$, we denote by $H(x, y) := \lim_{n\to\infty}P^{n}(x, y)/P^{n}(e, y)$ the ratio limit kernels for $x, y\in\Gamma$. We implicitly work under the assumption that these limits exist, in which case we say that $(\Gamma,\mu)$ have the strong ratio-limit property (SRLP). The (reduced) ratio-limit compactification $\Delta^r_{RL}\Gamma$ of the random walk $(\Gamma, \mu)$ is defined as the compactification of $\Gamma / R_{\mu}$ with respect to the ratio limit kernel functions $\{y\mapsto H(x, y)\}_{x\in \Gamma}$, where $R_\mu = \{y\in\Gamma \,|\, H(x, y) = H(x, e) \mbox{ for all } x\in\Gamma \}$ is a subgroup of $\Gamma$, so that the ratio limit kernels $y\mapsto H(x,y)$ naturally factor through $\Gamma / R_\mu$ for every $x\in \Gamma$.

Our inspiration stems from the work of Lalley \cite[Proposition 10]{Lalleyspacetime} where he identified the space-time Martin boundary of lazy nearest neighbor random walks on free groups as ``hollow cylinders with a top cap''. Our first result shows that the ratio-limit compactification can be naturally embedded as part of this ``top cap'' of the space-time boundary. 

\vspace{6pt plus 2pt minus 2pt}
\noindent\textbf{Theorem \ref{t:ratiolimitspacetime}.}
\textit{Let $\Gamma$ be a discrete group, and let $\mu$ be an admissible lazy symmetric finitely supported probability measure on $\Gamma$ such that the associated random walk of $(\Gamma, \mu)$ has SRLP. Then, there exists a natural embedding of $\Delta^r_{RL} \Gamma$ into the space-time Martin boundary.}
\vspace{6pt plus 2pt minus 2pt}

From Lalley's work \cite[Theorem 2]{Lalleyspacetime} we know that for lazy nearest neighbor random walk on the free group $\F_{d}$ of $d$ generators, the space-time Martin boundary is given by $\partial_{\ST}\F_{d} \cong ([0,R)\times \partial\F_{d}) \sqcup (\{R\} \times \Delta \F_{d})$ where the topology on the right-hand-side is the product topology induced from $[0,R]\times \partial \F_{d}$. On the other hand, by \cite[Proposition 10]{Lalleyspacetime} the \emph{minimal} space-time Martin boundary is given by $\partial^{m}_{\ST}\F_{d} \cong [0,R]\times \partial\F_{d}$. In order to calculate the minimal space-time boundary in greater generality, we introduce the so-called $0$-Martin boundary $\partial_{M, 0}\Gamma$. We define a (possibly non-symmetric) norm $|\cdot|_\mu$ on $\Gamma$ by declaring that
\begin{equation*}
    |x|_\mu := \inf \{n\geq 0 \,|\, P^{n}(e, x) > 0 \}.
\end{equation*}

We define the truncated family of measures $\bar \mu_x$ on $\Gamma$ by declaring that for $x, y\in \Gamma$, if $|y|_\mu\leq |x|_\mu$, then $\bar \mu_x(y) = 0$ and if $|y|_\mu\geq |x|_\mu + 1$, then $\bar \mu_x(y) = \mu(x^{-1}y)$. If $\mu(x^{-1}y) > 0$, then we necessarily have $|y|_\mu\leq |x|_\mu + 1$, hence $\bar\mu_x(y) > 0$ if and only if $|y|_\mu = |x|_\mu + 1$ and $\mu(x^{-1}y) > 0$.

Note that the sub-Markov chain $(Z_n)_{n\in \N}$ whose distribution is given by $(\bar\mu_x)_x$ is the random walk $(X_n)_{n\in \N}$ killed when reaching a point that could have been visited earlier. Roughly speaking $(Z_n)_{n\in \N}$ follows the law of $\mu$ but is allowed to move forward only when it does so with respect to the word semi-norm $|\cdot|_\mu$. The $0$-Martin boundary is defined as the Martin boundary of this sub-Markov chain with respect to the $0$-Martin kernels given by $K(x, y\,|\,0) = P^{|y| - |x|}(x, y)/P^{|y|}(e, y)$ when $e, x, y$ are aligned in this order on a geodesic and by $K(x, y\,|\,0) = 0$ otherwise. In Proposition \ref{p:0-stabilityinthegroup}, we show that at the level of the group, the $0$-Martin kernels arise as limits of the rescaled $\lambda$-Martin kernels as $\lambda \to 0$. More precisely, if we write $\tilde K(x, y \,|\, \lambda) := \lambda^{|x|}K(x, y \,|\, \lambda)$ when $\lambda \in (0,\rho^{-1}]$ and $\tilde K(x, y \,|\, 0) := K(x, y \,|\, 0)$, then we have that $\tilde K(x,y \, | \, \lambda) \rightarrow \tilde K(x,y \, | \, 0)$ as $\lambda \rightarrow 0$. 

Lalley's work shows that for lazy nearest neighbor random walks on free groups the $0$-Martin compactification is homeomorphic to the end compactification. We show that for random walks on hyperbolic groups, the $0$-Martin boundary naturally covers Gromov boundary, and that this cover is generally not injective.

\vspace{6pt plus 2pt minus 2pt}

\noindent\textbf{Theorem \ref{p:zeroMartintoGromov} + Proposition \ref{STcounterexample}.}
\textit{Let $\Gamma$ be a hyperbolic group, and let $\mu$ be a finitely supported admissible symmetric probability measure on $\Gamma$. Then, there exist a natural equivariant surjective continuous map from the $0$-Martin compactification $\Delta_{M, 0}\Gamma$ to the Gromov compactification $\Delta\Gamma$ which extends the identity on $\Gamma$. However, this map may fail to be injective in general.}
\vspace{6pt plus 2pt minus 2pt}

Our main result in this paper relates the minimal space-time Martin boundary to minimal $\lambda$-Martin boundaries for general lazy random walks with finite range.

\vspace{6pt plus 2pt minus 2pt}
\noindent\textbf{Theorem \ref{t:minimalspacetime}.}
\textit{Let $\Gamma$ be a discrete group, and let $\mu$ be an admissible finitely supported lazy probability measure on $\Gamma$. The minimal space-time Martin boundary is homeomorphic to the disjoint union of the minimal $\lambda$-Martin boundaries for $\lambda\in [0,\rho^{-1}]$, i.e. 
}
\begin{equation*}
    \partial^{m}_{\ST}\Gamma \cong \bigsqcup_{\lambda\in [0, \rho^{-1}]}\partial_{M, \lambda}^{m}\Gamma,
\end{equation*}
where the topology on the right-hand-side is the one induced from the point-wise convergence topology of functions $(\lambda, \xi) \mapsto \widetilde{K}(x,\xi \, | \, \lambda)$ for $x\in \Gamma$. 

\vspace{6pt plus 2pt minus 2pt}

More precisely, the point-wise convergence topology on $\bigsqcup_{\lambda\in [0, \rho^{-1}]}\partial_{M, \lambda}^{m}\Gamma$ is defined as the topology in which a sequence $(\lambda_{n}, \xi_{n})$ converges to the point $(\lambda, \xi)$ if and only if $\lambda_{n}\to\lambda$ and $\widetilde{K}(x, \xi_{n} \,|\, \lambda_{n}) \to \widetilde{K}(x, \xi \,|\, \lambda)$.

Finally, we describe the application of our results to operator algebras. Specifically, we use Theorem \ref{t:minimalspacetime} to characterize the non-commutative Shilov boundary (i.e. $C^{\ast}$-envelope) of the tensor algebra associated to a random walk on a group. Non-commutative Shilov boundary was first defined and studied by Arveson's in several papers \cite{ARWI69, ARWI72}, where it was first shown to exist by Hamana in \cite{HAMA79} through the use of injective envelopes of operator systems.

We first illustrate our approach in the context of classical Toeplitz C*-algebra. Let $U\colon \ell^{2}(\N)\to\ell^{2}(\N)$ be the unilateral shift given by $U(e_{n}) = e_{n + 1}$ for the standard orthonormal basis $(e_{n})_{n\in\N}$ of $\ell^{2}(\N)$. The $C^{\ast}$-algebra generated by $U$ is called the (classical) Toeplitz algebra $\mathcal{T}$, while the norm-closed algebra $\mathcal{T}^{+}$ generated by $U$ can be identified with the disk algebra $A(\mathbb{D})$ of continuous functions on $\overline{\mathbb{D}}$ which are analytic in its interior. It turns out that the closed two-sided ideal of $\mathcal{T}$ generated by $I-UU^{\ast}$ is isomorphic to the algebra of compact operators $\mathbb{K}(\ell^{2}(\N))$. Moreover, the quotient by this ideal is $\ast$-isomorphic to the space of continuous functions on the unit circle $C(\mathbb{T})$. This yields the following exact short sequence
\begin{equation} \label{eq:toeplitz-ext}
    0\to \mathbb{K}(\ell^{2}(\N))\to \mathcal{T} \to C(\mathbb{T}) \to 0.
\end{equation}
By maximum modulus principle, the quotient $\mathcal{T} \rightarrow C(\mathbb{T})$ is isometric on the disk algebra $A(\mathbb{D})$, and $C(\mathbb{T})$ is the smallest $C^{\ast}$-algebra quotient of $\mathcal{T}$ with this property. This comprises our first example of a C*-envelope, namely that of $A(\mathbb{D})$. 

This classical construction of Toeplitz C*-algebra, its quotients and its canonical non-self-adjoint subalgebra has been generalized in several different directions: among them are Cuntz-Krieger algebras arising from matrices with values in $\{0, 1\}$ \cite{CUKR80}, graph $C^{\ast}$-algebras for directed graphs in \cite{DOSA18}, Cuntz-Pimsner algebras associated to $C^{\ast}$-correspondences in \cite{PIMV97}, and Cuntz-Pimsner algebras associated to subproduct systems in \cite{VIAM12}; each one directly generalizing the previous setting. The construction of operator algebras arising from random walks naturally fits in the context of operator algebras arising from subproduct systems (see, e.g. \cite[Section 5]{DOAD21}), and provides a fertile ground for addressing some natural questions put forth by Viselter in \cite[Section 6]{VIAM12}. These questions have been taken up by several authors, specifically in works on $KK$-theory of $C^{\ast}$-algebras arising from Temperley-Leib subproduct systems by Habbestad and Neshveyev \cite{HANE23}, and by Arici, Gerontogiannis, and Neshveyev \cite{ARGN25}.

In the context of finite stochastic matrices, the $C^{\ast}$-envelope has been calculated by Dor-On and Markiewicz in \cite{DOMA17} (see also \cite{CDHLZ21}) where it was shown that the $C^{\ast}$-envelope quotient of the Toeplitz C*-algebra associated to the stochastic matrix can sometimes land strictly between the Toeplitz C*-algebra and Viselter's analogue of the Cuntz-Pimsner quotient. For random walks on non-trivial finite groups, the results in \cite{DOMA17} show that the $C^{\ast}$-envelope coincides with the Toeplitz $C^{\ast}$-algebra of the random walk. We will extend this to random walks on infinite discrete groups in Section \ref{sec:tensor_algebras}, and show that the $C^{\ast}$-envelope coincides with the Toeplitz $C^{\ast}$-algebra as well. 

Our strategy is to start with an exact sequence analogous to one in equation \eqref{eq:toeplitz-ext} and use Arveson's theory of boundary representations to characterize the $C^{\ast}$-envelope via strongly peaking representations \cite{ARWI11}. Theorem \ref{t:minimalspacetime} is then used to prove that all irreducible subrepresentations of the identity representation are strongly peaking. This allows us to answer \cite[Question 4]{VIAM12} in new examples by characterizing the $C^{\ast}$-envelope of the natural tensor algebra of a subproduct system arising from random walks.

\vspace{6pt plus 2pt minus 2pt}
\noindent\textbf{Theorem \ref{thm:envelope:main}.} \textit{Let $\Gamma$ be a discrete finitely generated group, and $\mu$ be an admissible lazy probability measure on $\Gamma$ with finite support. Then, the $C^{\ast}$-envelope of the tensor algebra of the random walk is the Toeplitz $C^{\ast}$-algebra of the random walk.}
\vspace{6pt plus 2pt minus 2pt}

\textbf{Organization.}
In Section \ref{s:preliminaries}, we provide the necessary background for the rest of the paper. In particular, we review the theory of harmonic functions and Martin boundaries for random walks on groups and operator algebras associated to them. In Section \ref{sec:space_time_martin}, we introduce the space-time Markov chain and show that the ratio limit compactification can be embedded in the space-time Martin compactification. In Section \ref{s:zeroMartin}, we introduce the $0$-Martin compactification and connect it to $\lambda$-Martin compactifications as $\lambda$ tends to $0$. In Section \ref{sec:hyperbolic}, we specialize to random walks on hyperbolic groups. We show that the $0$-Martin compactification covers the Gromov compactification and provide examples where the canonical covering map is not injective. In Section \ref{s:minimalspacetime}, we prove Theorem \ref{t:minimalspacetime}, which relates the minimal space-time Martin boundary to the minimal $\lambda$-Martin boundaries. Finally, in Section \ref{sec:tensor_algebras}, we prove Theorem \ref{thm:envelope:main} which characterizes the $C^{\ast}$-envelope of the tensor algebra associated to a random walk, identifying it with the corresponding Toeplitz $C^{\ast}$-algebra of the random walk.

\subsection*{Acknowledgments} The authors are deeply grateful to Matthieu Dussaule for sharing his notes and ideas on space-time boundaries. The authors are also grateful to Wolfgang Woess for bringing to our attention the work of Molchanov on minimal $\alpha$-harmonic functions on space-time Markov chains on Cartesian products \cite{Molchanov}.

\section{Preliminaries}
\label{s:preliminaries}

\subsection{Harmonic functions and $\lambda$-Martin boundaries.}
In this subsection we discuss some of the necessary preliminaries on random walks and their Martin boundaries. We refer to \cite{DYEU69} for the general boundary theory of transient sub-Markov chains, to \cite{KSKN76} for transient Markov chains accessible from a reference point, and to \cite{Woessbook} for irreducible Markov chains and random walks.

We first discuss classical $\lambda$-Martin boundaries for irreducible Markov chains induced by a probability measure on a countable discrete group $\Gamma$, which are known as random walks. Let $\Gamma$ be a countable discrete group and $\mu$ an admissible probability vector on $\Gamma$, i.e. so that the support $\supp(\mu) := \{x \in \Gamma: \mu(x) > 0\}$ generates $\Gamma$ as a semigroup. We denote by $P$ the \emph{irreducible} Markov kernel defined by $P(x, y) = \mu(x^{-1}y)$. We will call such a pair $(\Gamma, \mu)$ a \emph{$\mu$-random walk} on $\Gamma$. We say that the random walk is symmetric if $\mu(x) = \mu(x^{-1})$ for every $x\in \Gamma$. Finally, we say that $\mu$ is finitely supported if $\supp(\mu)$ is finite. 

Consider the \emph{$\lambda$-Green functions} given by
\begin{equation*}
    G(x, y \,|\, \lambda) := \sum_{n\geq 0}\lambda^n P^{n}(x, y),
\end{equation*}
which are defined for every $\lambda$ in an open ball of radius $R$ around the origin where 
\begin{equation*}
    R := \left(\limsup_{n\to\infty} \sqrt[n]{P^n(x, y)} \right)^{-1}
\end{equation*}
is the radius of convergence of the Green power series $\lambda \mapsto G(x,y \, | \, \lambda)$. Since random walks are irreducible Markov chains, the radius of convergence $R$ is independent of the choice of $x, y$ in $\Gamma$. Thus, it follows that $R=\rho^{-1}$ is the inverse of the spectral radius of the measure $\mu$ as defined in the introduction. We may define $\lambda$-Martin kernels for $0 < \lambda < R$ and $\lambda = R$, respectively, via
\begin{equation*}
    K(x, y \,|\, \lambda) := \frac{G(x, y \,|\, \lambda)}{G(x_{0}, y \,|\, \lambda)}\quad \mbox{and}\quad K(x, y \, | \, R) := \lim_{\lambda \rightarrow R}\frac{G(x, y \,|\, \lambda)}{G(x_{0}, y \,|\, \lambda)},
\end{equation*}
where $x, y\in\Gamma$. The second limit defining $K(x, y \, | \, R)$ exists by \cite[Lemma 3.66]{woessbook2} regardless of whether $G(x, y \, | \, R)$ is finite or not. In fact, by Varapoulos (see \cite[Theorem~7.8]{Woessbook}), the only groups that carry an admissible probability measure such that $G(x,y \, | \, R)$ is infinite (also known as $R$-recurrence) are groups with at most quadratic polynomial growth, i.e. groups that are virtually $\mathbb Z^d$ for $d\leq 2$. In particular, not only is $G(x,y \,|\, \lambda)$ always finite for $\lambda\in (0, R)$, it is also automatically finite at $\lambda = R$ if $\Gamma$ is not virtually $\mathbb Z^d$ for $d\leq 2$. We use the $\lambda$-Martin kernels to define the $\lambda$-Martin boundary for a random walk $(\Gamma, \mu)$ for $\lambda \in (0,R]$. 
\begin{definition}
    Given $0 < \lambda \leq R$, the compactification of $\Gamma$ with respect to the $\lambda$-Martin kernels $\{y \mapsto K(x, y \,|\, \lambda)\}_{x\in \Gamma}$ is called the \emph{$\lambda$-Martin compactification} and denoted by $\Delta_{M, \lambda}\Gamma$. The set difference of $\partial_{M, \lambda}\Gamma := \Delta_{M, \lambda}\Gamma \setminus \Gamma$ is called the $\lambda$-Martin boundary of $\Gamma$.
\end{definition}

As before, a sequence $y_n\in\Gamma$ converges to a point $\xi\in \partial_{M, \lambda}\Gamma$ if and only if for every $x\in\Gamma$ we have that $K(x, y_n \,|\, \lambda)$ converges to the value $K(x, \xi \,|\, \lambda)$ of the limit $\lambda$-Martin kernel.

For random walks, the $\lambda$-Martin compactification always exists and is a metrizable space. It can be constructed as the completion of $\Gamma$ endowed with the distance $d_{M, \lambda}$ defined by
\begin{equation*}
    d_{M, \lambda}(y_1, y_2) := \sum_{x\in \Gamma}\alpha(x)\frac{\big|K(x, y_1 \,|\, \lambda) - K(x, y_2 \,|\, \lambda)\big| + \big|\delta_x(y_1) - \delta_x(y_2)\big|}{C_x + 1},
\end{equation*}
where $C_x>0$ is a constant satisfying $K(x, y \,|\, \lambda)\leq C_x$ for all $x, y$ (see, e.g. \cite[Lemma 4.1]{SAST97}) and $\sum_{x\in \Gamma}\alpha(x)$ is some convergent series.

The term $\delta_x(y_1) - \delta_x(y_2)$ within the definition of the metric guarantees that $\Gamma$ is open in the Martin compactification.
In fact, it could happen that there is a sequence $y_n$ going to infinity and a point $y$ in $\Gamma$ such that for all $x$, $K(x, y_n \,|\, \lambda)$ converges to $K(x, y \,|\, \lambda)$.
In such a case, there would also be a point $\xi$ in the boundary such that $K(x, y \,|\, \lambda) = K(x, \xi \,|\, \lambda)$ for all $x$.
Thus, the term $\delta_x(y_1) - \delta_x(y_2)$ is needed in the definition of the metric above to construct an actual compactification of $\Gamma$, see \cite[equation (4.4)]{SAST97} and also the comments before \cite[Definition~24.5]{Woessbook}. This shows that the map $y\in \Delta_{M, \lambda}\Gamma\mapsto K(\cdot, y \,|\, \lambda)$ might not be one-to-one. Moreover, although this holds by definition if $y_n\in \Gamma$ and $\xi$ is in the boundary, generally a sequence $y_n\in \Delta_{M, \lambda} \Gamma$ may fail to converges to a point $y\in \Delta_{M, \lambda} \Gamma$ even if $K(x, y_n \,|\, \lambda)$ converges to $K(x, y \,|\, \lambda)$ for all $x$. We will see that such pathologies can be avoided by assuming that $\mu$ is finitely supported.

Recall that for every $x \in \Gamma$, there is a constant $C_x>0$ such that $K(x, y \,|\, \lambda)\leq C_x$ for every $y \in \Gamma$. We denote by $\mathcal{B}^{+}_{1}(\Gamma, \R)$ the space of non-negative functions $f\colon\Gamma\to\R_{+}$ such that $f(e) = 1$ and $f(x)\leq C_x$ for every $x$. By Tychonov's theorem, $\mathcal{B}^{+}_{1}(\Gamma, \R)$ is a compact Hausdorff space with the topology of point-wise convergence.

\begin{proposition}
\label{p:embeddingMartinintoboundedfunctions}
    
    The map $\xi\in \partial_{M, \lambda}\Gamma\mapsto K(\cdot, \xi \,|\, \lambda)\in\mathcal{B}^{+}_{1}(\Gamma, \R)$ is one-to-one and continuous.
    Moreover, if $\mu$ is finitely supported, then the map $y\in \Delta_{M, \lambda}\Gamma\mapsto K(\cdot, y \,|\, \lambda)$ is also one-to-one and continuous.
\end{proposition}

\begin{proof}

    By definition of the $\lambda$-Martin boundary, if for two points $\xi_1, \xi_2\in \partial_{M, \lambda}\Gamma$, if we have $K(x, \xi_1 \,|\, \lambda) = K(x, \xi_2 \,|\, \lambda)$ for all $x$, then $\xi_1 = \xi_2$, so the map is one-to-one.
    Now assume that $\xi_n\in \partial_{M, \lambda}\Gamma$ converges to $\xi\in \partial_{M, \lambda}\Gamma$.
    This implies that for every $x \in \Gamma$,
    \begin{equation*}
        \alpha(x)\frac{\big|K(x, \xi_n \,|\, \lambda) - K(x, \xi \,|\, \lambda)\big|}{C_x + 1}\underset{n\to \infty}{\longrightarrow}0
    \end{equation*}
    and so $K(x, \xi_n \,|\, \lambda)$ converges to $K(x, \xi \,|\, \lambda)$, so that $K(\cdot, \xi_n \,|\, \lambda)$ converges to $K(\cdot, \xi \,|\,\lambda)$ in the point-wise convergence topology. Thus, $\xi \mapsto K(\cdot,\xi \, | \, \lambda)$ is continuous.
    
    The same proof shows that the map $y\in \Delta_{M, \lambda}\Gamma\mapsto K(\cdot, y \,|\, \lambda)$ is continuous.
    What is left to do is to prove that whenever $\mu$ is finitely supported, then this map is also one-to-one. First of all, we have $\sum_{x\in \Gamma}\lambda\mu(x)G(x, y \,|\, \lambda) = G(e, y \,|\, \lambda) - \delta_e(y)$. This shows that for $y\in \Gamma$, the function $K(\cdot, y \,|\, \lambda)$ is $\lambda^{-1}$-harmonic everywhere except for at $y$.
    In particular, if $y_1\neq y_2$ in $\Gamma$, then $K(\cdot, y_1 \,|\, \lambda)\neq K(\cdot, y_2 \,|\, \lambda)$. Now, if $\mu$ is finitely supported, then letting $y$ tend to infinity, we have that $K(\cdot, \xi \,|\, \lambda)$ is $\lambda^{-1}$-harmonic for $\xi$ in the boundary, see \cite[Lemma~24.16]{Woessbook}. Consequently, in this case, for every $y\in \Gamma$ and every $\xi\in \partial_{M, \lambda}\Gamma$, we also have $K(\cdot, y \,|\, \lambda)\neq K(\cdot, \xi \,|\, \lambda)$, which concludes the proof.
\end{proof}

\begin{corollary}
\label{c:pointwiseconvergenceMartin}

    A sequence $\xi_n\in \partial_{M, \lambda}\Gamma$ converges to a point $\xi\in \partial_{M, \lambda}\Gamma$ if and only if the sequence $K(x, \xi_n \,|\, \lambda)$ converges to $K(x, \xi \,|\, \lambda)$ for all $x\in\Gamma$. If moreover, $\mu$ is finitely supported, then a sequence $y_n \in \Delta_{M, \lambda}\Gamma$ converges to a point $y \in \Delta_{M, \lambda}\Gamma$ if and only if $K(x, y_n \,|\, \lambda)$ converges to $K(x, y \,|\, \lambda)$ for all $x\in\Gamma$.
\end{corollary}

\begin{proof}

    By the first part of Proposition~\ref{p:embeddingMartinintoboundedfunctions}, the map $\xi\in \partial_{M, \lambda}\Gamma\mapsto K(\cdot, \xi \,|\, \lambda)\in\mathcal{B}^{+}_{1}(\Gamma, \R)$ is continuous and one-to-one. Since both the $\lambda$-Martin boundary and the space $\mathcal{B}^{+}_{1}(\Gamma, \R)$ are compact Hausdorff, this implies that it is a homeomorphism onto its image, which conclude the proof.
    If $\mu$ is finitely supported, then we use the second part of Proposition~\ref{p:embeddingMartinintoboundedfunctions}.
\end{proof}

Similar to the classical case, $\lambda$-Martin boundaries are intimately related to harmonic functions. 

\begin{definition}

    A function $f\colon\Gamma\to \R$ is called \emph{$t$-harmonic} if $P_\mu f = tf$, where $P_\mu$ is the Markov operator associated with $\mu$, that is, for all $x\in\Gamma$,
    \begin{equation*}
        P_\mu f(x) = \sum_{y\in \Gamma}P(x,y)f(y) = tf(x).
    \end{equation*}
\end{definition}
If there exists a $t$-harmonic function, then we necessarily have $t\geq R^{-1}$, see \cite[Lemma~7.2]{Woessbook}. We will prefer the notation $\lambda^{-1}$-harmonic functions for $\lambda\leq R$. We denote by $\mathcal{H}^{+}_1{}(P, \lambda)$ the set of non-negative $\lambda^{-1}$-harmonic functions $h$ such that $h(e) = 1$ with pointwise convergence topology. If $\mu$ is admissible, then the minimum principle holds, see \cite[(1.15) Minimum Principle]{Woessbook}, hence any non-negative $\lambda^{-1}$-harmonic function is strictly positive. We have the following representation theorem, see \cite[Theorem~24.7]{Woessbook}.

\begin{theorem}[Martin-Poisson representation theorem]
\label{t:MartinPoissonrepresentation}

    For every $h\in\mathcal{H}^{+}_1{}(P, \lambda)$, there exists a probability measure $\nu_h$ on $\partial_{M, \lambda}\Gamma$ such that for all $x\in \Gamma$, we have
    \begin{equation*}
        h(x) = \int_{\partial_{M,\lambda}\Gamma}K(x, \xi \,|\, \lambda)d\nu_h(\xi).
    \end{equation*}
\end{theorem}

Assuming that $\mu$ is finitely supported, then $K(\cdot, \xi \,|\, \lambda)$ is itself $\lambda^{-1}$-harmonic, see \cite[Lemma~24.16]{Woessbook}, but in general, it might not be. Also, \textit{a priori} there is no reason for the measure $\nu_h$ in Theorem~\ref{t:MartinPoissonrepresentation} to be unique. To ensure uniqueness, we need to consider the minimal Martin boundary.

\begin{definition}
\label{d:minimality}

    A non-negative $\lambda^{-1}$-harmonic function $h$ with $h(e) = 1$ is called \emph{minimal} if for every other non-negative $\lambda^{-1}$-harmonic function $g$ with $g\leq Ch$ for some constant $C>0$, we have in fact $g = ch$ for a constant $c>0$. The minimal $\lambda$-Martin boundary is the set
    \begin{equation*}
        \partial^m_{M, \lambda}\Gamma = \{\xi\in \partial_{M, \lambda}\Gamma \,:\, K(\cdot, \xi \,|\, \lambda) \text{ is minimal }\lambda^{-1}\text{-harmonic}\}.
    \end{equation*}
\end{definition}

The minimal $\lambda$-Martin boundary corresponds to the set of extreme points in $\mathcal{H}_+(\lambda)$, which is a compact convex set with the pointwise convergence topology. We then have that for every non-negative $\lambda^{-1}$-harmonic function $h$ with $h(e) = 1$, the probability measure $\nu_h$ given by Theorem~\ref{t:MartinPoissonrepresentation} can be chosen uniquely if we require it to have full measure on $\partial_{M, \lambda}^m\Gamma$ (see \cite[Theorem~24.9]{Woessbook}).

\medskip
At $\lambda = 0$, we have
$G(x, y \,|\, 0) = \delta_{x, y}$, i.e. $G(x, y) = 1$ if $x = y$ and $0$ otherwise. Thus, it is not sensible to use the Green function to define the Martin boundary at $\lambda = 0$. We will provide full details for the suitable notions of $0$-Martin boundary and $\infty$-harmonic functions in Section~\ref{s:zeroMartin}.

For the purpose of defining and studying the space-time boundary and $0$-Martin boundary in Sections \ref{sec:space_time_martin} and \ref{s:zeroMartin}, we will need the theory of Martin boundaries for possibly reducible sub-Markov chains that satisfy an accessibility condition. 

Consider a countable space $\Omega$ endowed with a sub-Markov transition kernel $P$, i.e.\ a function $P\colon \Omega\times\Omega\to \mathbb R_{\geq 0}$ such that $\sum_{y\in \Omega}P(x, y) \leq 1$ for all $x\in\Omega$. This defines a sub-Markov chain $X := (X_n)_{n}$ whose step distribution is given by $P$. The \textit{Green function} of the chain is given for $x, y\in\Omega$ by
\begin{equation*}
    G(x, y) := \sum_{n\geq 0}P^{n}(x, y).
\end{equation*}
The radius of convergence of the Green series $\lambda \mapsto \sum_{n\geq 0}\lambda^n P^{n}(x, y)$ is given via
\begin{equation*}
    R(x, y) := \left(\limsup_{n\to\infty} \sqrt[n]{P^n(x, y)} \right)^{-1}.
\end{equation*}
A sub-Markov chain $X$ is called \textit{transient} if $G(x, y)$ is finite for all $x, y\in\Omega$. We say that the sub-Markov chain $X$ is a \textit{accessible} (or \emph{accessible from a base-point}) if there exists a base-point $x_{0}\in\Omega$ such that for each $y\in\Omega$, there exists $n\in\N$ with $P^{n}(x_{0}, y) > 0$. When $X$ is additionally transient, we can readily define the \textit{Martin kernels} by $K(x, y) := G(x, y)/G(x_{0}, y)$ for all $x, y\in\Omega$.

\begin{definition}
    Let $X$ be an accessible transient sub-Markov chain on a countable state space $\Omega$ with base-point $x_0$. The minimal compactification of the state space $\Omega$ with respect to the Martin kernels $\{y \mapsto K(x, y)\}_{x\in\Omega}$ is called the \emph{Martin compactification} of $\Omega$ according to base-point $x_0$ and denoted $\Delta_M\Omega$. The set difference of $\partial_M\Omega := \Delta_M\Omega\setminus\Omega$ is called the \emph{Martin boundary} of $\Omega$.
\end{definition}

The Martin compactification is the unique smallest compact topological space containing $\Omega$ as a discrete open and dense subset so that the functions $y \mapsto K(x, y)$ extend continuously. Similarly to before, when the sub-Markov chain is transient and accessible, then the Martin kernels are bounded by \cite[Lemma 4.1]{SAST97}, and we may embed the state space $\Omega$ inside a compact product of intervals with the product topology. The closure of $\Omega$ in that product compactum is a realization of the Martin compactification. 

It follows form the definition of a Martin compactification that a sequence $(y_n)_{n}$ in $\Omega$ converges to a point $\xi\in\partial_M\Omega$ if and only if we have that $K(x, y_n)$ converges to $K(x, \xi)$ for every $x\in\Omega$. In other words, the Martin functions $x\mapsto K(x,y_n)$ converge in the point-wise topology to $x\mapsto K(x,\xi)$. Moreover, we have the following description of convergence on the boundary similarly to Corollary \ref{c:pointwiseconvergenceMartin}.

\begin{proposition}
\label{prop:sub-Markov:Martin:convergence}

    Let $X$ be an accessible transient sub-Markov chain on a countable state space $\Omega$ with base-point $x_0$. A sequence $\xi_{n}\in \partial_{M}\Omega$ converges to a point $\xi\in \partial_{M}\Omega$ if and only if $K(x, \xi_n)$ converges to $K(x, \xi)$ for all $x\in \Omega$.
\end{proposition}

\begin{proof}

    The proof follows verbatim the proofs of Corollary \ref{c:pointwiseconvergenceMartin} and Proposition \ref{p:embeddingMartinintoboundedfunctions} with $\lambda = 1$.
\end{proof}

\begin{definition}

    A function $h\colon\Omega\to \mathbb R$ is called \emph{harmonic} if for all $x\in\Omega$,
    \begin{equation*}
        Ph(x) = \sum_{y\in \Omega}P(x, y)h(y) = h(x),
    \end{equation*}
    where $P$ is the sub-Markov operator. When $h$ is non-negative and harmonic with $h(x_{0}) = 1$, we say that $h$ is \emph{minimal} if for every other non-negative harmonic function $g$ with $g\leq h$, we have in fact that $g = Ch$ for a constant $C > 0$. We then define the \emph{minimal} Martin boundary as the set
    \begin{equation*}
        \partial^{m}_{M}\Omega = \{\xi\in \partial_{M}\Omega \,:\, K(\cdot, \xi) \text{ is minimal harmonic}\}.
    \end{equation*}
\end{definition}

We now recall the Martin-Poisson representation theorem for sub-Markov chains that are transient and accessible from \cite[Theorem 6]{DYEU69}.
\begin{theorem}[Martin-Poisson representation theorem]
\label{thm:Martin-Poisson:sub-Markov}

    Let $X$ be an accessible transient sub-Markov chain on a countable state space $\Omega$ with base-point $x_0 \in \Omega$. Then, for every non-negative harmonic function $h$ such that $h(x_{0}) = 1$, there exists a unique probability measure $\nu_h$ on $\partial^{m}_M\Omega$ such that for all $x\in\Omega$,
    \begin{equation*}
        h(x) = \int_{\partial^{m}_M\Omega}K(x, \xi)d\nu_h(\xi).
    \end{equation*}
\end{theorem}

\begin{proof}

    The theorem follows from the result of Dynkin on superharmonic functions in \cite[Theorem 6]{DYEU69}. Note that the ``space of exits" in \cite{DYEU69} corresponds exactly to superharmonic functions with the unique representing measure, and that the respective Martin kernels are necessarily harmonic by \cite[Theorem 5]{DYEU69}.
\end{proof}

Even though the Martin kernel $x \mapsto K(x, \xi)$ is harmonic for any point $\xi \in \partial_M^m \Omega$, the results in \cite{DYEU69} do not necessarily assume that $x \mapsto K(x, \xi)$ is harmonic for all $\xi \in \partial_M \Omega$.

\begin{proposition}
\label{p:all-bnry-harm}
    
    Let $X$ be an accessible transient sub-Markov chain on a countable state space $\Omega$ with base-point $x_0 \in \Omega$. If $P$ has finite range, i.e. for each $x\in\Omega$, there exists finitely many $y\in\Omega$ such that $P(x, y) > 0$, then the limit Martin kernels $K(\cdot, \xi)$ are harmonic for every $\xi\in\partial_{M}\Omega$.
\end{proposition}

\begin{proof}

    The argument is identical to that used in \cite[Lemma 7.6]{Woessbook}.
\end{proof}

Thus, for a sub-Markov kernel with finite range, we can apply the representation theorem to functions $x \mapsto K(x,\xi)$ for any $\xi \in \partial_M \Omega$. We will exploit this fact later in Section~\ref{sec:tensor_algebras}. 

In Section \ref{sec:hyperbolic}, we will need the fact that an extension of the identity map between two compact Hausdorff spaces is automatically continuous whenever it preserves convergence of sequences from the group.
\begin{proposition}
\label{prop:extension_of_identity}

    Let $\Gamma$ be a countable discrete set densely embedded into two compact Hausdorff spaces $X$ and $Y$, and let $f\colon X\to Y$ be an extension of identity map on $\Gamma$, i.e. $f\restriction_{\Gamma} = \id_{\Gamma}$. If for every net $(x_{\alpha})_{\alpha}$ in $\Gamma$ converging to $\xi\in X$, the net $(f(x_{\alpha}))_{\alpha}$ converges to $f(\xi)\in Y$, then $f$ is continuous.
\end{proposition}

\begin{proof}

    Fix a net $(x_{\alpha})_{\alpha}$ in $X$ converging to $\xi\in X$. We want to show that $(f(x_{\alpha}))_{\alpha}$ converges to $f(\xi)\in Y$. Since $Y$ is compact, it will suffice to show that for every convergent subnet of $(f(x_{\alpha}))_{\alpha}$ we have that the subnet must converges $f(\xi)$. So, we assume without loss of generality that $(f(x_{\alpha}))_{\alpha}$ converging to some point $\eta\in Y$, and we will show that $\eta = f(\xi)$.
    
    Take neighborhoods $U$ of $\xi$ in $X$ and $V$ of $\eta$ in $Y$, and choose $\alpha$ large enough such that $x_{\alpha}$ are in $U$ and $f(x_{\alpha})$ are in $V$. Observe that since $\Gamma$ is dense in $X$, the intersection $\Gamma\cap U\cap V$ is not empty. Thus, for each $\alpha$ we may choose a point $x_{\alpha, U, V}\in\Gamma\cap U\cap V$ and define a directed partial order on the triples $(\alpha, U, V)$ by using the partial order on $\alpha$'s and inverse inclusion for $U$ and $V$. Then, since $f$ is an extension of the identity map on $\Gamma$, the net $(x_{\alpha, U, V})_{\alpha, U, V}$ in $\Gamma$ converges to both $\xi$ in the topology on $X$ and tp $\eta$ in the topology on $Y$. Thus, by assumption, since this net converges to $\xi$ in $X$, it must also converge to $f(\xi)$ in $Y$. Since $Y$ is Hausdorff, we get that $\eta = f(\xi)$.
\end{proof}

\vspace{6pt plus 2pt minus 2pt}

\subsection{Operator algebras and representations.} \label{ss:op-alg-prelim} An operator algebra $\mathcal{A}$ is a norm-closed subalgebra of bounded operators $\mathbb{B}(\mathcal{H})$ on the Hilbert space $\mathcal{H}$. If, furthermore, $\mathcal{A}$ is closed under adjoints, we say that $\mathcal{A}$ is a $C^{\ast}$-algebra. We denote by $\mathbb{B}(\mathcal{H})$ and $\mathbb{K}(\mathcal{H})$ the $C^{\ast}$-algebra of all bounded operators and of all compact operators, respectively. The latter is given by the norm closure of finite rank operators on $\mathcal{H}$. When $\A$ is a C*-algebra, we will denote by $\sp(\A)$ the irreducible representation spectrum of $\A$, which is the set of equivalence classes of irreducible representations of $\A$ up to unitary equivalence.

Next we will discuss some of the basic theory of operator algebras and operator systems relevant to our paper. We refer the reader to \cite{PAVE03} for additional details, and to \cite{ARWI69, ARWI08} for a more in-depth exposition to the theory of non-commutative boundaries. 

Recall that an operator system $\S$ is a unital $\ast$-closed subspace of $\B(\mathcal{K})$ on some Hilbert space $\mathcal{K}$. We say that a linear map $\pi\colon \S\to \B(\H)$ \textit{positive} if it sends the cone of positive elements to the cone of positive elements, and \textit{completely positive} if its matrix ampliations 
\begin{equation}
    \begin{split}
        \pi^{(n)}\colon M_{n}(\S) & \to M_n(\mathbb{B}(\H)), \\
        [s_{ij}]^{n}_{i, j = 1} & \mapsto [\pi(s_{ij})]^{n}_{i, j = 1}
    \end{split}
\end{equation}
are positive for each $n\in\N$, where we identify $M_n(\mathbb{B}(\H)) \cong \mathbb{B}(\H^{\oplus n})$ (so that positivity and the norm are understood). We will say that a unital completely positive map $\pi\colon S\to \B(\H)$ has the \textit{unique extension property} if $\pi$ has a \emph{unique} unital completely positive extension $\tilde{\pi}\colon C^{\ast}(\S)\to \B(\H)$ such that $\tilde{\pi}$ is multiplicative. If, moreover, $\tilde{\pi}$ is an irreducible $*$-representation, we call $\tilde{\pi}$ a \textit{boundary} representation. 

Boundary representations of operator systems are crucial since they allow one to minimally recover the matrix norms on $\mathcal{S}$, providing us with a non-commutative analogue of the Krein-Milman theorem from classical convexity \cite{DAKE15}.

\begin{theorem}[{\cite[Theorem 7.1]{ARWI08}, \cite[Theorem 3.4]{DAKE15}}]
\label{tm:norm_via_bdry_rep}
    Let $\mathcal{S}$ be an operator system. Then, for $n\in\N$ and $[s_{ij}]_{i, j}\in M_{n}(\mathcal{S})$, ranging over all boundary representations $\tilde{\pi} : C^*(\S) \rightarrow \B(\H)$ for $\mathcal{S}$, we have that
    \begin{equation}
        \|[s_{ij}]\| = \sup_{\tilde{\pi}}\|[\tilde{\pi}(s_{ij})]\|.
    \end{equation}
\end{theorem}

As our goal is to compute the $C^{\ast}$-envelope of the tensor algebra of a random walk, so we recall several definitions from non-commutative boundary theory of operator algebras. For an operator algebra $\mathcal{A}$, we call a pair $(\mathcal{B}, \iota)$ a \textit{$C^{\ast}$-cover} of $\mathcal{A}$ if $\mathcal{B}$ is a $C^{\ast}$-algebra and $\iota\colon\mathcal{A}\to \mathcal{B}$ is a completely isometric homomorphism such that $C^{\ast}(\iota(\mathcal{A})) = \mathcal{B}$. There always exists a unique smallest $C^{\ast}$-cover called the \textit{$C^{\ast}$-envelope} of $\mathcal{A}$ and denoted $(C^{\ast}_{\env}(\mathcal{A}), \kappa)$, which satisfies the following universal property: if $(\mathcal{B}, \iota)$ is any other $C^{\ast}$-cover of $\mathcal{A}$, then the map $\iota(a)\mapsto\kappa(a)$ extends uniquely to a surjective $\ast$-homomorphism $\pi\colon\mathcal{B}\to C^{\ast}_{\env}(\mathcal{A})$ such that $\pi\circ\iota = \kappa$. 

Given an operator algebra $\mathcal{A}$ and a C*-cover $(\iota,\mathcal{B})$, we will often suppress $\iota$, and consider $\A$ as a subalgebra of $\mathcal{B} = C^*(\A)$. We will call an ideal $\J \lhd \mathcal{B}$ such that the quotient map $q\colon \mathcal{B}\to \mathcal{B}/ \J$ restricts to a completely isometric map on $\mathcal{A}$ a \textit{boundary} ideal for $\mathcal{A}$. If, additionally, $\J \lhd \mathcal{B}$ is an ideal which contains every other boundary ideal for $\mathcal{A}$, we call $\J$ the \textit{Shilov} ideal for $\mathcal{A}$ in $\mathcal{B}$. It turns out that when $\J$ is the Shilov ideal, the natural quotient map $q\colon \mathcal{B}\to \mathcal{B}/\J$ realizes $\mathcal{B}/\I$ as the $C^{\ast}$-envelope of $\mathcal{A}$ (see, e.g. \cite{HAMA79}).

We may always turn an operator algebra $\mathcal{A}$ into an operator system by first unitizing the operator algebra $\mathcal{A}^1$, and then generating the operator system $\mathcal{A}^1 + (\mathcal{A}^1)^*$ from it. By Meyer's theorem \cite[Corollary 3.2]{MERA01}, the unitization $\mathcal{A}^1$ is independent of the initial completely isometric embedding of $\mathcal{A}$ in $\mathbb{B}(\K)$, and by \cite[Proposition 2.12]{PAVE03} the operator system generated by a unital operator algebra is also independent of embedding. This process of moving from operator algebras to operator systems is functorial in the sense that for any completely contractive representation $\pi\colon \mathcal{A} \rightarrow \mathbb{B}(\H)$, by \cite[Corollary 3.3]{MERA01} there is always a uniquely determined unital completely contractive homomorphism $\pi^1\colon \mathcal{A}^1 \rightarrow \mathbb{B}(\H)$, which in turn promotes by \cite[Proposition 3.5]{PAVE03} to a unital completely positive map $\widetilde{\pi^1}\colon \mathcal{A}^1 + (\mathcal{A}^1)^* \rightarrow \mathbb{B}(\H)$. This allows us to apply non-commutative boundary theory of operator systems to the theory of (potentially non-unital) operator algebras, as has been fleshed out by Dor-On and Salomon in \cite[Proposition 2.4, 2.5]{DOSA18}. For instance, we have corresponding definitions for the unique extension property and boundary representations. A completely contractive homomorphism $\pi\colon \mathcal{A} \to \B(\H)$ has the \textit{unique extension property} if $\pi$ has a unique contractive completely positive extension $\tilde{\pi}\colon C^{\ast}(\mathcal{A})\to \B(\H)$, and this $\tilde{\pi}$ is multiplicative. In case $\widetilde{\pi}$ is irreducible, we still call $\tilde{\pi}$ a \textit{boundary representation}. Thus, when discussing potentially non-unital operator algebras, for the purpose of studying boundary representations, we may reduce any problem about boundary representations with respect to operator algebras, to a problem about boundary representations with respect to operator systems.

The $C^{\ast}$-envelope is often referred to as the non-commutative Shilov boundary whereas the set of all boundary representations is referred to as a non-commutative Choquet boundary. They comprise the non-commutative analogues of the respective notions from classical function theory (see, e.g. \cite{GATE05}). Theorem \ref{tm:norm_via_bdry_rep} is then the non-commutative counterpart of density of Choquet boundary in Shilov boundary, and whenever wee are given a completely isometric representation $\pi\colon \mathcal{A} \rightarrow \mathbb{B}(\H)$ with the unique extension property, we have realize the C*-envelope of $\A$ as the C*-algebra $C^*(\pi(\mathcal{A}))$ generated by the image of $\pi$. 

We will work with completely strongly peaking irreducible representations as studied by Arveson in \cite{ARWI11}. An irreducible $*$-representation $\pi\colon C^{\ast}(\mathcal{A})\to\mathbb{B}(\H)$ is \textit{completely strongly peaking} for $\A$ if there exists $n\in\N$ and an operator matrix $T = [T_{ij}]\in M_{n}(\mathcal{A})$ such that
\begin{equation}
\label{eq:completely_peaking}
    \|\pi^{(n)}(T)\| > \sup_{\rho\nsim\pi}\|\rho^{(n)}(T)\|,
\end{equation}
where the supremum is taken over all irreducible representations $\rho$ inequivalent to $\pi$, denoted $\rho\nsim\pi$. We will be able to circumvent having to work directly with boundary representations since we will work with irreducible representations supported on ideals of compact operators. In this case, by \cite[Theorem 7.2]{ARWI11} we know that such irreducible representations are boundary if and only if they are completely strongly peaking representations.

We now discuss the construction of the tensor and Toeplitz algebras for a random walk. Let $\Gamma$ be a countable discrete group, $\mu$ be an admissible probability measure on $\Gamma$, and $P$ the associated Markov kernel. We denote by $\Gr(P)$ the set of edges of the random walk on $\Gamma$, which are exactly the pairs $(x,y)$ for which $P(x,y)>0$. For $z\in\Gamma$, we denote the \textit{space-time set} at a base point $z$,
\begin{equation*}
    \ST_{z} := \{(y, m)\in\Gamma\times\Z_{+} \,:\, P^{m}(y, z) > 0\}.
\end{equation*}

\begin{remark}
Note that while the notation seems similar to that of $\ST$, the space-time set $\ST$ is equal to the space-time set $\ST_e$ at the base point $e$ only for random walks with symmetrically supported measures. The importance in distinguishing these two sets has to do with the orientation chosen for the operator algebras associated with random walks, and will become apparent later on.
\end{remark}

 We will call the Hilbert space
\begin{equation*}
    \mathcal{H} = \mathcal{H}(\Gamma, \mu) := \bigoplus_{z\in\Gamma}\ell^{2}(\ST_{z})
\end{equation*}
the \textit{Fock Hilbert space} of the random walk $(\Gamma, \mu)$, where $\{e^{(m)}_{x, z}\}_{(x, m)\in\ST_{z}}$ is the standard orthonormal basis of the space
\begin{equation*}
    \mathcal{H}_{z} = \mathcal{H}_{z}(\Gamma, \mu) := \ell^{2}(\ST_{z})
\end{equation*}
for each $z\in\Gamma$. We next define the space-time shift operators
\begin{equation*}
    \begin{split}
        S^{(n)}_{x, y}\colon \mathcal{H}(\Gamma, \mu) & \to\mathcal{H}(\Gamma, \mu), \\
        e^{(m)}_{y', z} & \mapsto\delta_{y, y'}\sqrt{\frac{P^{n}(x, y)P^{m}(y, z)}{P^{n + m}(x, z)}}e^{(n + m)}_{x, z}
    \end{split}
\end{equation*}
for $n, m\in\N$ and $P^{n}(x, y) > 0$, where the adjoints are easily shown to be given by
\begin{equation*}
    (S^{(n)}_{x, y})^{\ast}(e^{(n + m)}_{x', z}) = \delta_{x, x'}\sqrt{\frac{P^{m}(x, y)P^{n}(y, z)}{P^{m + n}(x, z)}}e^{(n)}_{y, z}.
\end{equation*}

\begin{definition}

    Let $\Gamma$ be a countable discrete group, and $\mu$ be an admissible probability measure on $\Gamma$. Then, 
    \begin{enumerate}[(i)]
        \item the $C^{\ast}$-algebra
            \begin{equation*}
                \mathcal{T}(\Gamma, \mu) := C^{\ast}(S^{(n)}_{x, y} \,|\, P^{n}(x, y) > 0,\, n\in\N) \subset \mathbb{B}(\mathcal{H}(\Gamma, \mu))
            \end{equation*}
            is called the \textit{Toeplitz algebra} associated with the random walk $(\Gamma, \mu)$;
        \item the norm-closed operator algebra
            \begin{equation*}
                \mathcal{T}^{+}(\Gamma, \mu) := \overline{\Alg}^{\|\cdot\|}(S^{(n)}_{x, y} \,|\, P^{n}(x, y) > 0,\, n\in\N) \subset \mathbb{B}(\mathcal{H}(\Gamma, \mu))
            \end{equation*}
            is called the \textit{tensor algebra} associated with the random walk $(\Gamma, \mu)$.
    \end{enumerate}
\end{definition}

Next, we define the relevant ideal of compact operators and the corresponding quotient by it of Toeplitz algebra.

\begin{definition}

    Let $\Gamma$ be a countable discrete group, and $\mu$ be a finitely supported admissible probability measure on $\Gamma$. Then, the C*-subalgebra of compact operators
    \begin{equation*}
        \mathcal{I}(\Gamma, \mu) := \bigoplus_{z\in \Gamma}\K(\mathcal{H}_{z}(\Gamma, \mu))
    \end{equation*}
    is called the \textit{Viselter ideal} of $\mathcal{T}(\Gamma, \mu)$, and the quotient $C^{\ast}$-algebra
    \begin{equation*}
        \mathcal{O} = \mathcal{O}(\Gamma, \mu) := \mathcal{T}(\Gamma, \mu)/\mathcal{I}(\Gamma, \mu)
    \end{equation*}
    is called the \textit{Cuntz--Pimsner algebra} associated with the random walk $(\Gamma, \mu)$.
\end{definition}

From \cite[Proposition 4.4]{DOAD21}, we see that $\mathcal{I}(\Gamma, \mu)$ is indeed an ideal in the Toeplitz algebra $\mathcal{T}(\Gamma, \mu)$ whenever $\mu$ is finitely supported. Similar to \cite[Section 3]{DOMA17}, using the following short exact sequence 
\begin{equation*}
    0\to \mathcal{I}(\Gamma, \mu) \to \mathcal{T}(\Gamma, \mu) \to \mathcal{O}(\Gamma, \mu) \to 0
\end{equation*}
and the discussion preceding \cite[Theorem 1.3.4]{ARWI76}, we see that given any representation $\rho$ of the Toeplitz algebra $\mathcal{T}(\Gamma, \mu)$, we have that $\rho$ decomposes uniquely into a direct sum $\rho = \rho_{\mathcal{I}}\oplus\rho_{\mathcal{O}}$ of representations, where $\rho_{\mathcal{I}}$ is the unique extension of $\rho|_{\mathcal{I}(\Gamma, \mu)}$ to $\mathcal{T}(\Gamma, \mu)$, and $\rho_{\mathcal{O}}$ annihilates $\mathcal{I}(\Gamma, \mu)$. If moreover $\rho$ is irreducible, then one of the summands is necessarily trivial.

\section{Space-time and Ratio-limit boundaries}
\label{sec:space_time_martin}

Let $\Gamma$ be a countable discrete group, and let $\mu$ be a finitely supported admissible probability measure on $\Gamma$. We construct the \textit{space-time set} $\ST$ as follows. Points in $\ST$ are the elements $(x, n)$ of $\Gamma\times\mathbb Z_+$ such that the random walk can reach $x$ from $e$ in $n$-th step, i.e. $P^n(e,x)=\mathbb P[X_n = x \, | \, X_0 = e] > 0$. We think of $\ST$ as a directed graph by adding an edge between $(x, m)$ and $(y, n)$ if $n = m + 1$ and $\mu(x^{-1}y) > 0$.

Another way of thinking about the directed graph on $\ST$ is to construct it inductively as follows. One starts with the vertex $(e, 0)$ and add an edge connecting $(e, 0)$ to all $(x, 1)$ such that $\mu(x) > 0$. Then, at time $m$, for all vertices $(x, m)$, one adds an edge connecting $(x, m)$ to all $(y, m + 1)$ whenever $\mu(x^{-1}y) > 0$.

We now define a Markov chain $(Y_k)_{k\in\Z_{+}}$ on $\ST$ by specifying a Markov kernel $P_{\ST}$ on $\ST$. That is, for all $(x, m),(y,n)\in \ST$, we define the Markov kernel $P_{\ST}$ with the state space $\ST$ via
\begin{equation*}
    P_{\ST}((x, m), (y, n))= \delta_{m + 1, n} \cdot \mu(x^{-1}y).
\end{equation*}
We call the associated Markov chain $(Y_k)_{k\in\Z_{+}}$ the \emph{space-time Markov chain}. Note that the set of pairs $((x,m),(y,n))$ for which $P_{\ST}((x, m), (y, n))>0$ are exactly the pairs corresponding to edges in the directed graph structure on $\ST$. By construction, we have that $P^n_{ST}((e,0),(y,n))>0$ for all $(y,n) \in \ST$, so that $G((e,0),(y,n))>0$ for any $(y,n) \in \ST$. Thus, the space-time Markov chain $(Y_k)_{k\in \Z_+}$ is a transient accessible Markov chain, and therefore admits a Martin compactification.

\begin{definition}
The space-time Martin boundary $\partial_{\ST}\Gamma$ of the random walk $(\Gamma,\mu)$ is the $1$-Martin boundary $\partial_M\ST$ of the space-time Markov chain $(Y_k)_{k\in\Z_{+}}$ on $\ST$.
\end{definition}

\medskip
In order to study the space-time Martin boundary more precisely, we first note that the Green function associated with the Markov chain $(Y_k)_{k \in \N}$ is given by
\begin{equation*}
    G((x, m), (y, n)) = \sum_{k\geq 0}P^k_{\ST}((x, m), (y, n)).
\end{equation*}
Since the Markov chain $(Y_k)_{k \in \N}$ can only move forward in the time factor $\mathbb Z_+$, the only possible $k$ such that
$P_{\ST}^k ((x, m), (y, n))) > 0$ is $k = n - m$.
Now, to reach $(y, n)$ from $(x, m)$ in $(n - m)$ steps, the chain $(Y_k)_{k \in \N}$ needs to exactly follow a trajectory for the random walk of length $n - m$ from $x$ to $y$. Thus, we find that
\begin{equation*}
    G((x, m), (y, n)) = P_{\ST}^{n - m}((x, m), (y, n)) = P^{n - m}(x, y).
\end{equation*}

For each $(x,m) \in \ST$ we denote by $K_{\ST}((x,m), \cdot)$ the extension to a continuous function on the space-time compactification $\Delta_{\ST} \Gamma$ given on the dense set $\ST$ by 
$$
K_{\ST}((x,m), (y,n)) = \frac{P^{n - m}(x, y)}{P^{n}(e, y)}.
$$

We then have the following reformulation of convergence to a point in the space-time boundary. A sequence $(y_k, n_k)$ in $\ST$ converges to a point $\xi$ in the space-time Martin boundary $\partial_{\ST}\Gamma$ if and only if for all $(x, m)$ in $\ST$,
    \begin{equation*}
        \frac{P^{n_k - m}(x, y_k)}{P^{n_k}(e, y_k)}\underset{k\to \infty}{\longrightarrow}K_{\ST}((x, m),\xi)
    \end{equation*}
In other words, we need to find all possible limits of ratios $P^{n - m}(x, y)/P^{n}(e, y)$ when $(y, n)$ goes to infinity on $\ST$. On the boundary, we have the following description of sequential convergence by Proposition \ref{prop:sub-Markov:Martin:convergence}. 

\begin{corollary} 
\label{c:pointwiseconvergencespacetime}

    A sequence $(\xi_n)_{n\in \N}$ in $\partial_{\ST} \Gamma$ converges to a point $\xi$ in $\partial_{\ST} \Gamma$ if and only if we have that $K_{\ST}((x, m), \xi_n) \rightarrow K_{\ST}((x, m), \xi)$ for all $(x, m) \in \ST$.
\end{corollary}

\medskip
\textbf{Standing assumptions.}
We will assume henceforth that $\mu$ is finitely supported and admissible, and, except if stated otherwise, we will assume in what follows that the random walk is \textit{lazy}, i.e. that $\mu(e) > 0$. This implies that for every $(x, m)\in \ST$ and $(x, n)\in\ST$ for every $n\geq m$.

We now recall the definition of the ratio limit compactification that was introduced in \cite{WOWO21} and \cite{DOAD21}. Let $\Gamma$ be a finitely generated group, and let $\mu$ be a probability measure on $\Gamma$. When the limit on the right hand side exists, define the ratio limit kernel as 
\begin{equation*}
    H(x, y) := \lim_{n\to\infty}\frac{P^{n}(x, y)}{P^{n}(e, y)}
\end{equation*}
for every $x, y$ in $\Gamma$. In general, we do not know if this limit exists, and we say that $(\Gamma,\mu)$ satisfies the strong ratio limit property (SRLP) if this limit exists for every $x,y\in \Gamma$ as $n$ tends to infinity. 

SRLP was already considered in the work of Chung and Erd\" os \cite{CE51}, followed by works by Kingman--Orey in \cite{KO64} and Kesten \cite{Kes63}. Moreover, we do know that the limit exists for instance in the context of amenable groups (see, e.g. \cite[Theorem 1]{Avezratiolimit}, \cite[Theorem 1.1]{DOSH24}). Beyond the situation of amenable groups, SRLP is automatically satisfied in the presence of a local limit theorem. For instance when we have the asymptotic behavior
\begin{equation*}
    P^n(e, x)\sim \beta(x)R^{-n}n^{-\alpha}
\end{equation*}
as $n\to \infty$, where $R$ is the inverse of the spectral radius of $\mu$. In many cases and in particular in the examples we study in this paper, such a local limit theorem holds. For instance, admissible aperiodic symmetric random walks with finite support on finitely generated hyperbolic groups admit a local limit theorem as above with exponent $\alpha = 3/2$ (see \cite[Theorem~1.1]{GouezelLLT}).

\medskip
The ratio limit compactification $\Delta_{RL}\Gamma$ is constructed similarly to the $\lambda$-Martin compactification by replacing $\lambda$-Martin kernels $K(x, y \,|\, \lambda)$ with ratio limit kernels $H(x, y)$. The remainder of $\Gamma$ in the ratio-limit compactification is called the ratio limit boundary and is denoted by $\partial_{RL}\Gamma$. Thus, up to homeomorphism $\Delta_{RL}\Gamma$ is the unique smallest compactum (up to homeomorphism extending the identity on $\Gamma$) to which the functions $y \mapsto H(x,y)$ extends continuously for every $x\in \Gamma$. A sequence $y_n$ in $\Gamma$ converges to a point $\xi\in\partial_{RL}\Gamma$ if and only if for every $x\in \Gamma$, we have that $H(x, y_n)$ converges to a the limit kernel $H(x, \xi)$ provided by the definition of the ratio-limit compactification.

We now discuss a separated version of the ratio-limit compactification. Let the \textit{ratio-limit radical} to be the set
\begin{equation*}
    R_\mu = \{y\in\Gamma \,|\, H(x, y) = H(x, e) \mbox{ for all } x\in\Gamma \}.
\end{equation*}
Then, $R_\mu$ is a subgroup of $\Gamma$, see \cite[Proposition~3.2]{DOAD21}. Moreover, if we assume that $\mu$ is symmetric, then it is a normal subgroup, see \cite[Proposition~2.4]{DDG+}, and the ratio limit kernels $H(x, y)$ factor through $R_\mu$.

\begin{definition}
The reduced ratio limit compactification $\Delta^{r}_{RL}\Gamma$ is the unique smallest compactification of $\Gamma/R_\mu$ that make the (well-defined) functions $y R_{\mu} \mapsto H(x, y)$ continuous. The reduced ratio limit boundary $\partial^{r}_{RL}\Gamma$ is the set difference of $\Delta^{r}_{RL}\Gamma \setminus [\Gamma/R_\mu]$.
\end{definition}

Our next goal is to compare the ratio limit compactification with the space-time Martin boundary in the presence of SRLP.

Fix $y\in \Gamma$, and let $n_k$ be a sequence that goes to infinity. Then for $(x,m)\in \ST$,
\begin{equation*}
    \frac{P^{n_k - m}(x, y)}{P^{n_k}(e, y)} = \frac{P^{n_k - m}(x, y)}{P^{n_k}(x, y)}\frac{P^{n_k}(x, y)}{P^{n_k}(e, y)},
\end{equation*}
so that,
\begin{equation*}
    \frac{P^{n_k}(x, y)}{P^{n_k}(e, y)}\underset{k\to \infty}{\longrightarrow}H(x, y).
\end{equation*}
By a result of Gerl \cite{Gerl1, Gerl2} (see also \cite{WOWO21}), for every $x, y\in\Gamma$, we have
\begin{equation*}
    \frac{P^{k + 1}(x, y)}{P^k(x, y)}\underset{k\to \infty}{\longrightarrow}R^{-1},
\end{equation*}
so that,
\begin{equation*}
    \frac{P^{n_k - m}(x, y)}{P^{n_k}(x, y)}\underset{k\to \infty}{\longrightarrow}R^m.
\end{equation*}
Therefore, for any $(x,m) \in \ST$, we get 
\begin{equation*}
    \frac{P^{n_k - m}(x, y)}{P^{n_k}(e, y)}\underset{k\to \infty}{\longrightarrow}R^{m}H(x, y),
\end{equation*}
where $n_k$ is a sequence going to infinity. Thus, we see that $(y, n_k)$ converges to a point $\Phi(y)$ in the space-time Martin boundary. Thus, we obtained a well-defined map $\Phi$ from $\Gamma$ to the space-time Martin boundary $\partial_{\ST}\Gamma$ by setting $\Phi(y)$ to be the point in $\partial_{\ST}\Gamma$ corresponding to the function $h_y\colon \ST \rightarrow \mathbb{R}_+$ given by $h_y(x, m) = R^mH(x, y)$.

Although $y\mapsto H(x,y)$ is in general not one-to-one, we see that $y \mapsto \Phi(y)$ is not one-to-one. However, by definition of $R_{\mu}$ we know that $\Phi$ factors through $R_\mu$ to yields a one-to-one map $\Tilde{\Phi}$ from $\Gamma/R_\mu$ to the space-time Martin boundary as we now show.

\begin{lemma}

    The map $\Tilde{\Phi}$ is an injective map from $\Gamma/R_{\mu}$ to $\partial_{\ST}\Gamma$.
\end{lemma}

\begin{proof}

    We will show that the map $\Phi$ satisfies that $\Phi(y_1) = \Phi(y_2)$ if and only if $y_1^{-1}y_2\in R_\mu$. Suppose that $\Phi(y_1) = \Phi(y_2)$. Then, for every $(x, m) \in \ST$, we have
    \begin{equation*}
        R^mH(x, y_1) = R^mH(x, y_2).
    \end{equation*}
    Since the random walk is admissible, for every $x\in\Gamma$, there exists $m$ such that $(x, m)\in \ST$. We deduce that for every $x\in \Gamma$, we have $H(x, y_1) = H(x, y_2)$. The co-cycle property for $H$ then yields that
    \begin{equation*}
        H(x, y_1^{-1}y_2) = \frac{H(y_1x, y_2)}{H(y_1, y_2)} = \frac{H(y_1x, y_1)}{H(y_1, y_1)} = H(x, e)
    \end{equation*}
    holds for every $x\in \Gamma$. Conversely, if $y_1^{-1}y_2\in R_\mu$, then by the same co-cycle property we get for every $x\in \Gamma$ that
    \begin{equation*}
        H(x, y_2) = H(x, y_1y_1^{-1}y_2) = \frac{H(y_1^{-1}x, y_1^{-1}y_2)}{H(y_1^{-1}, y_1^{-1}y_2)} = \frac{H(y_1^{-1}x, e)}{H(y_1^{-1}, e)} = H(x, y_1).
    \end{equation*}
    Therefore, for every $(x, m)\in \ST$, we get
    \begin{equation*}
        R^mH(x, y_1) = R^mH(x, y_2),
    \end{equation*}
    so that $\Phi(y_1) = \Phi(y_2)$, which concludes the proof.
\end{proof}

Our next goal is to show that $\Tilde{\Phi}$ is continuously extendable to the reduced ratio limit compatification $\Delta_{RL}^r\Gamma$. We recall that by \cite[Lemma 15(ii)]{WOWO21} there exists a continuous surjection 
\begin{equation*}
    \Psi\colon \Delta_{RL}\Gamma \twoheadrightarrow \Delta^r_{RL}\Gamma
\end{equation*}
from the full to the reduced ratio limit compactification which extends the canonical quotient map $\Gamma \rightarrow \Gamma / R_{\mu}$. Since $\Phi = \Tilde{\Phi} \circ \Psi$ on $\Gamma$, if we show that $\widetilde{\Phi}$ is continuously extendable to the reduced ratio limit compactification $\Delta_{RL}^r\Gamma$, it will automatically follow that $\Phi = \Tilde{\Phi} \circ \Psi$ is continuously extendable to $\Delta_{RL}\Gamma$. 

\begin{proposition}
There exists a continuous extension of $\Tilde{\Phi}$ to $\Delta^{r}_{RL}\Gamma$. Therefore, $\Phi$ continuously extends to $\Delta_{RL}\Gamma$ via the equation $\Phi = \Tilde{\Phi} \circ \Psi$.
\end{proposition}

\begin{proof}
    
    Let $\xi\in \partial^{r}_{RL}\Gamma$. As $\Gamma / R_{\mu}$ is dense in $\Delta^{r}_{RL}\Gamma$, there is a sequence $y_k\in\Gamma/R_{\mu}$ converging to $\xi$. This implies that for all $x\in \Gamma$,
    \begin{equation*}
        H(x, y_k)\underset{k\to \infty}{\longrightarrow}H(x, \xi).
    \end{equation*}
    Now, let $h_k$ be the limit function defined by $\Tilde{\Phi}(y_k)$. By the discussion above,
    \begin{equation*}
        h_k((x, m)) = R^{m}H(x, y_k)
    \end{equation*}
    for every $(x, m)\in \ST$. Therefore, for every such fixed $(x, m)$ we have
    \begin{equation*}
        h_k((x, m))\underset{k\to \infty}{\longrightarrow}R^mH(x, \xi).
    \end{equation*}
    In other words, the sequence of functions $h_k$ converges point-wise to the function $h_\xi$ given by $h_{\xi}(x,m) = R^m H(x,\xi)$. By definition of point-wise convergence of functions representing points in space-time Martin boundary, the limit $h_\xi$ must correspond to a limit point of the space-time Martin kernels. We denote by $\Tilde{\Phi}(\xi)$ the corresponding point in the space-time Martin boundary. Note that this is well-defined since for any choice of a sequence $y_k$ converging to $\xi$ in $\Delta^{r}_{RL}\Gamma$, the limit function $H(x, \xi)$ will be independent of $y_k$, so that the limit function $h_\xi$ is also independent of $y_k$. Thus, we are left to show that the extension (which we continue to denote by) $\Tilde{\Phi}$ is continuous. 
    
    Suppose $\xi_n\in \Delta^{r}_{RL}\Gamma$ converges to $\xi$. This implies that for all $x \in \Gamma$ that $H(x, \xi_n)$ converges to $H(x, \xi)$. Now, the limit of space-time Martin kernels defined by $\Tilde{\Phi}(\xi_n)$ is given by
    \begin{equation*}
        h_n\colon (x,m)\mapsto R^{m}H(x, \xi_n)
    \end{equation*}
    and the limit of space-time Martin kernels defined by $\Tilde{\Phi}(\xi)$ is given by
    \begin{equation*}
        h_{\xi}\colon (x, m)\mapsto R^{m}H(x, \xi).
    \end{equation*}
    Thus, it follows readily that $h_n$ converges point-wise to $h_{\xi}$, so that by definition of point-wise convergence of functions representing points in the space-time Martin boundary we get that $\Tilde{\Phi}(\xi_n)$ converges to $\Tilde{\Phi}(\xi)$.
\end{proof}

We arrive at the main result of this section, showing that the reduced ratio-limit compactification naturally embeds inside the space-time boundary.

\begin{theorem}
\label{t:ratiolimitspacetime}

    Let $\Gamma$ be a countable discrete group, and let $\mu$ be an admissible finitely supported lazy symmetric probability measure on $\Gamma$ such that $(\Gamma, \mu)$ has SRLP. Then, there exists a continuous map $\Phi\colon\Delta_{RL}\Gamma \to \partial_{\ST}\Gamma$ and a homeomorphism onto its image $\Tilde{\Phi}\colon\Delta^{r}_{\RL}\Gamma \to \partial_{\ST}\Gamma$ which make the following diagram commute
    \[
        \begin{tikzcd}
            \Delta_{RL}\Gamma \arrow[r, rightarrow, "\Phi"] \arrow[d, twoheadrightarrow, "\Psi"] & \partial_{\ST}\Gamma \\
            \Delta^{r}_{RL}\Gamma \arrow[ru, hookrightarrow, "\Tilde{\Phi}"{below}] & 
        \end{tikzcd}.
    \]
\end{theorem}

\begin{proof}

    It remains to show that the continuous extension $\Tilde{\Phi}$ to $\Delta_{RL}^r\Gamma$ is still an embedding. Let $\xi_1, \xi_2 \in \Delta_{RL}^r \Gamma$, and assume that $\Tilde{\Phi}(\xi_1) = \Tilde{\Phi}(\xi_2)$. Then, letting $h_1$ and $h_2$ be the limits of the space-time Martin kernels corresponding to $\Tilde{\Phi}(\xi_1)$ and $\Tilde{\Phi}(\xi_2)$, we have that
    \begin{equation*}
        h_i((x, m)) = R^{m}H(x, \xi_i).
    \end{equation*}
    In particular, this implies that $H(x, \xi_1) = H(x, \xi_2)$ for all $x\in\Gamma$, so that $\xi_1 = \xi_2$. Hence, we see that $\Tilde{\Phi}$ is one-to-one. Finally, since the reduced ratio limit compactification $\Delta^{r}_{RL}\Gamma$ is compact and Hausdorff, the injective map $\Tilde{\Phi}$ is automatically a homeomorphism onto its image.
\end{proof}

\section{The 0-Martin boundary}
\label{s:zeroMartin}

Let $\Gamma$ be a discrete countable group, and let $\mu$ be an admissible lazy finitely supported probability measure on $\Gamma$. We define the following (possibly non-symmetric) norm $|\cdot|_\mu$ on $\Gamma$ by declaring that
\begin{equation*}
    |x|_\mu = \inf \{ n\geq 0 \,|\, P^{n}(e, x)>0 \}.
\end{equation*}
Since $P^{0}(e, x) = \delta_e(x)$, we have that $|x|_\mu = 0$ if and only if $x = e$.
Also, note that since
\begin{equation*}
    P^{|x|_\mu + |y|_\mu}(e, xy)\geq P^{|x|_\mu}(e, x)P^{|y|_\mu}(e, x^{-1}xy)>0,
\end{equation*}
we get $|xy|_\mu\leq |x|_\mu + |y|_\mu$.
In fact, $|\cdot|_\mu$ is the word distance for the (possibly non-symmetric) finite generating set given by the support of $\mu$. In the following, we call $|\cdot|_\mu$ the word distance induced by the support of $\mu$.

We introduce the truncated family of measures $\bar \mu_x$ on $\Gamma$ by declaring that for $x, y\in \Gamma$, that if $|y|_\mu\leq |x|_\mu$ then $\bar \mu_x(y) = 0$ and if $|y|_\mu\geq |x|_\mu + 1$ then $\bar \mu_x(y) = \mu(x^{-1}y)$.
If $\mu(x^{-1}y) > 0$, then we necessarily have $|y|_\mu\leq |x|_\mu + 1$, so that $\bar\mu_x(y) > 0$ if and only if $|y|_\mu = |x|_\mu + 1$ and $\mu(x^{-1}y) > 0$.

Note that the sub-Markov chain $Z = (Z_{n})_{n}$, whose distribution is given by $(\bar \mu_x)_x$, is the random walk $X$ killed when reaching a point that could have been visited earlier. Roughly speaking $Z$ follows the law of $\mu$ but is assigned to only move forward with respect to the word distance $|\cdot|_\mu$. To simplify notations, we will write $|\cdot|$ instead of $|\cdot|_\mu$ in this section, as there will be no other word distance involved.

\begin{lemma}
    The Green function $G_{\bar\mu}$ associated with $\bar\mu$ is given by
    \begin{equation*}
        G_{\bar\mu}(x, y) = P^{|y| - |x|}(x, y)
    \end{equation*}
    if $e,x,y$ are aligned in this order on a geodesic for the word distance $|\cdot|$, and by $G_{\bar\mu}(x, y) = 0$ otherwise.
\end{lemma}

\begin{proof}

    By definition,
    \begin{equation*}
        G_{\bar\mu}(x, y) = \sum_{n\geq 0}\sum_{z_1, \dots, z_n}\bar\mu_x(z_1)\bar\mu_{z_1}(z_2)\dots\bar\mu_{z_n}(y).
    \end{equation*}
    Since $\bar\mu_u(v)$ is positive if and only if $|v| = |u| + 1$ and $\mu(u^{-1}v)>0$, we see by induction that the sum can only be positive if and only if there exists some $n\geq 0$ such that $|y| = |x|+n$, so that $n = |y| - |x|$ is uniquely determined. But now, if $e,x,y$ are not aligned in this order on a word geodesic, then $|x| + n = |y| = |xx^{-1}y| < |x| + |x^{-1}y|$ so that $n < |x^{-1}y| = \inf \{ \, m\geq 0 \, | \, P^m(x,y)>0 \, \}$. On the other hand, we clearly have $P^n(x,y)>0$, which yields a contradiction. Thus, $e,x,y$ are aligned in this order on a word geodesic if and only if $G_{\bar\mu}(x, y)>0$. Since
    \begin{equation*}
        \sum_{z_1, \dots, z_n}\bar\mu_x(z_1)\bar\mu_{z_1}(z_2)\dots\bar\mu_{z_n}(y)
    \end{equation*}
    can be positive if and only if $n = |y| - |x|$, we see that
    \begin{equation*}
        \begin{split}
            G_{\bar\mu}(x, y) & = \sum_{z_1, \dots, z_{|y| - |x|}}\bar\mu_{x}(z_1)\bar\mu_{z_1}(z_2)\dots\bar\mu_{z_{|y| - |x|}}(y) \\
            & = \sum_{z_1, \dots, z_{n}}\mu(x^{-1}z_1)\mu(z_1^{-1}z_2)\dots\mu(z_n^{-1}y) \\
            & = P^{|y| - |x|}(x, y),
        \end{split}
    \end{equation*}
    which concludes the proof.
\end{proof}

Consequently, the Martin kernels associated with $\bar\mu$ are given by
\begin{equation*}
    K_{\bar\mu}(x, y) = \frac{P^{|y| - |x|}(x, y)}{P^{|y|}(e, y)}
\end{equation*}
if $e, x, y$ are aligned in this order on a geodesic and by $K_{\bar\mu}(x, y) = 0$ otherwise. Our next result is the main reason for introducing the family of measures $\bar\mu_x$.
\begin{proposition}
\label{p:0-stabilityinthegroup}

    For every fixed $x, y\in \Gamma$, we have that
    \begin{equation*}
        \lambda^{|x|}K(x, y \,|\, \lambda)\to K_{\bar\mu}(x, y)
    \end{equation*} 
    as $\lambda\to 0$.
\end{proposition}

\begin{proof}

    Fix $x, y\in \Gamma$, and denote $m := |x^{-1}y|$, so that $m$ is the minimal integer such that $P^{m}(x, y) > 0$, and $n := |y|$. Then, we have
    \begin{equation*}
        G(x, y \,|\, \lambda) = \sum_{k\geq 0}\lambda^kP^{k}(x, y) = \sum_{k\geq m}\lambda^kP^{k}(x, y) = \lambda^m\sum_{k\geq m}\lambda^{k - m}P^{k}(x, y).
    \end{equation*}
    Consequently, as $\lambda\to 0$, we get
    \begin{equation*}
        G(x, y \,|\, \lambda)\sim \lambda^mP^{m}(x, y)
    \end{equation*}
    and
    \begin{equation*}
        G(e, y \,|\, \lambda)\sim \lambda^nP^{n}(e, y).
    \end{equation*}
    That is, as $\lambda \to 0$, we have that
    \begin{equation*}
        K(x, y \,|\, \lambda)\sim \lambda^{m - n}\frac{P^{m}(x, y)}{P^{n}(e, y)}.
    \end{equation*}
    Now, assume that $e, x, y$ are not aligned in this order on a geodesic for the word distance $|\cdot|$, so that $m= |x^{-1}y| > |y| - |x| = n - |x|$. Therefore, since $|x| > n - m$, we see that $\lambda^{|x|}K(x, y \,|\, \lambda)$ converges to $0$ as $\lambda$ tends to $0$. On the other hand, if we assume that $e, x, y$ are aligned in this order on a geodesic, we get that $m = |x^{-1}y| = |y| - |x|= n - |x|$. Thus,
    \begin{equation*}
        \lambda^{|x|}K(x, y \,|\, \lambda)\sim \frac{P^{|y| - |x|}(x, y)}{P^{|y|}(e, y)} = K_{\bar\mu}(x, y)
    \end{equation*}
    as $\lambda$ tends to $0$, which concludes the proof.
\end{proof}

\begin{definition}

    A function $g\colon\Gamma\to \mathbb R$ is called \emph{$\infty$-harmonic} if it is harmonic for the family of finite measures $(\bar\mu_x)_{x\in\Gamma}$, i.e. for every $x\in \Gamma$,
    \begin{equation*}
        g(x) = \sum_{y\in \Gamma}\bar\mu_x(y)g(y).
    \end{equation*}
    It is said to be \emph{normalized} if $g(e) = 1$. A normalized $\infty$-harmonic function $g$ is called \emph{minimal} if for every other normalized $\infty$-harmonic function $h$ such that $h\leq Cg$ for some $C > 0$, we have in fact $h = Cg$.
\end{definition}

Henceforth, we denote $K(\cdot, \cdot \,|\, 0) := K_{\bar\mu}(\cdot, \cdot)$ the $0$-Martin kernel.
\begin{definition}
The $1$-Martin compactification of $\Gamma$ with respect to the $0$-Martin kernels is called the \emph{$0$-Martin compactification} and denoted by $\Delta_{M, 0}\Gamma$. The \emph{$0$\nobreakdash-Martin boundary} is then the set difference $\partial_{M, 0}\Gamma:= \Delta_{M, 0}\Gamma \setminus \Gamma$. We denote by $K(\cdot, \xi \,|\, 0)$ the function corresponding to the point $\xi\in \partial_{M, 0}\Gamma$, and by $\partial^{m}_{M, 0}\Gamma$ the corresponding minimal Martin boundary.
\end{definition}

Proposition~\ref{p:0-stabilityinthegroup} is interpreted as saying that at the level of the group, the $0$-Martin kernels arise as limits of the $\lambda$-Martin kernels with appropriate rescaling as $\lambda\to 0$. More precisely, as $\lambda$ tends to $0$, the relevant function to study is $\lambda^{|x|}K(x, y \,|\, \lambda)$ rather than $K(x, y \,|\, \lambda)$. We will thus use the following notations below. We will henceforth write $\tilde K(x, y \,|\, \lambda) := \lambda^{|x|}K(x, y \,|\, \lambda)$ if $\lambda > 0$ and $\tilde K(x, y \,|\, 0) := K(x, y \,|\, 0)$. Proposition~\ref{p:0-stabilityinthegroup} shows that for fixed $x,y \in \Gamma$, as $\lambda \rightarrow 0$ we get that $\tilde K(x,y \, | \, \lambda) \rightarrow \tilde K (x,y \, | \, 0)$. However, this does not guarantee that the $0$-Martin boundary arises as a limit set of $\lambda$-Martin boundaries as $\lambda\to 0$.

\section{0-Martin boundaries of hyperbolic groups}
\label{sec:hyperbolic}

At the extremity of $\lambda = 0$ in the interval $[0,\rho^{-1}]$, we will exhibit a new kind of behavior. In this section, we will assume that $\mu$ is symmetric in order to ensure that the word distance $|\cdot|_\mu$ introduced in Section~\ref{s:zeroMartin} actually defines a (symmetric) metric on $\Gamma$. However, we will not assume that $\mu$ is lazy since in this section we will only be concerned with the $0$-Martin boundary.

\begin{proposition}
\label{p:zeroMartintoGromov}

    Let $\Gamma$ be a hyperbolic group, and let $\mu$ be an admissible finitely supported symmetric probability measure on $\Gamma$. Then, there exists an equivariant surjective continuous map $\psi_\mu$ from the $0$-Martin compactification onto the Gromov compactification which extends the identity map on $\Gamma$.
\end{proposition}

\begin{proof}

    Let $h$ be a non-negative $\infty$-harmonic function with $h(e) = 1$. Then, $h(x)$ is positive for an infinite number of $x\in \Gamma$. Indeed, we prove by induction that for every $n \geq 0$, there exists $x\in\Gamma$ with $|x| = n$ such that $h(x) > 0$.
        
    When $n = 0$, we take $x = e$ since $h(e) = 1$. Now, assume that the property holds for some $n\in \mathbb{N}$. We take $x\in \Gamma$ such that $|x| = n$ with $h(x) > 0$. Then, since $h$ is $\infty$-harmonic, we have
    \begin{equation*}
        h(x) = \sum_{y\in \Gamma}\overline{\mu}_x(y)h(y) > 0,
    \end{equation*}
    so there exists $y\in \Gamma$ with $\overline{\mu}_x(y)h(y) > 0$. The fact that both $\overline{\mu}_x(y) > 0$ and $h(y) > 0$ implies that $|y| = n + 1$ and $h(y) > 0$, which proves the claim.

    Now, we define an extension $\psi_{\mu}$ of the identity map on $\Gamma$ to points $\xi\in \partial_{M,0}\Gamma$. Let $(y_k)_{k}$ be a sequence in $\Gamma$ converging to $\xi\in \partial_{M,0}\Gamma$. Denote by $h_\xi(\cdot) = K(\cdot, \xi \, | \, 0)$ the corresponding $0$-Martin function. Then, since $h_{\xi}$ is $\infty$-harmonic, we have that $h_\xi(x) > 0$ for infinitely many $x\in\Gamma$. Let $x$ be such an element. Since by Proposition \ref{p:0-stabilityinthegroup}, we have that $K(x, y_k \, | \, 0) \rightarrow h_{\xi}(x) > 0$ as $k\rightarrow \infty$, we see that there exists $k_0$ such that for all $k\geq k_0$, we have $P^{|y_k| - |x|}(x, y_k) > 0$. Thus, for $k \geq k_0$, we have that $x$ is on a geodesic from $e$ to $y_k$ for the word metric. This proves that geodesics from $e$ to $y_k$ fellow travel up to $x$, and since $|x|$ is arbitrarily large, we see that $y_k$ converges to a point in the Gromov boundary, which we denote by $\psi_\mu(\xi)$.
    
    We now prove that $\psi_\mu$ is well-defined. Let us assume that $(y_k^{(1)})_{k}$ and $(y_k^{(2)})_{k}$ are two sequences converging to the same point $\xi$ in the $0$-Martin boundary. Then, for every $x\in\Gamma$ for which $h_{\xi}(x) > 0$, there exists some $k_0$ such that for all $k\geq k_0$, by the same argument as above, geodesics from $e$ to $y_k^{(i)}$ fellow travel up to $x$ for all $k\geq k_0$, and therefore the limits of $(y_k^{(1)})_{k}$ and $(y_k^{(2)})_{j}$ define the same point in the Gromov boundary. We have shown that for any sequence $(y_k)_{k}$ which converges to a point $\xi\in \partial_{M, 0}\Gamma$, the very same sequence $(y_k)_{k}$ converges to a unique point $\psi_{\mu}(\xi)$. Thus, by Proposition \ref{prop:extension_of_identity} we deduce that $\psi_{\mu}$ is continuous.
    
    Since $\psi_{\mu}$ is clearly equivariant (as it is so on $\Gamma$), we are left with showing that $\psi_{\mu}$ is surjective. Let $\zeta$ be a point in the Gromov boundary and let $(y_k)_{k}$ be a sequence in $\Gamma$ converging to $\zeta$ in the Gromov boundary. Then, up to a sub-sequence, $(y_k)_{k}$ converges to a point $\xi$ in the $0$-Martin boundary, and so, necessarily, $\psi_\mu(\xi) = \zeta$.
\end{proof}

In Example \ref{ex:free:0-martin}, we will see that the map $\psi_{\mu}$ above can fail to be injective in general. For the purpose of showing this, we develop a bit of theory.

Since we assume that $\mu$ is finitely supported, the $0$-Martin kernel $x \mapsto K(x , \xi \, | \, 0)$ is $\infty$-harmonic for every $\xi \in \partial_{M,0}\Gamma$ by Proposition \ref{p:all-bnry-harm}. Thus, by Theorem \ref{thm:Martin-Poisson:sub-Markov}, there exists a unique probability measure $\nu_\xi$ on the minimal $0$-Martin boundary such that
\begin{equation*}
    K(x, \xi \,|\, 0) = \int_{\partial_{M,0}^m\Gamma}K(x, \zeta |0)d\nu_\xi(\zeta)
\end{equation*}
for each $x\in\Gamma$.

It turns out that the support of the measure is the preimage of the corresponding point in Gromov boundary, as the following proposition demonstrates.

\begin{proposition}
For every $\xi\in \partial_{M,0}\Gamma$, the measure $\nu_\xi$ is supported on $\psi_\mu^{-1}(\psi_\mu(\xi))$.
\end{proposition}

\begin{proof}
Notice that for any $\xi$ in the 0-Martin boundary, $x\mapsto K(x, \xi \,|\, 0)$ can only be positive on the union of geodesics from $e$ to $\psi_\mu(\xi)$. We denote by $\mathcal{P}(\xi)$ the set of $x\in \Gamma$ for which $K(x, \xi \,|\, 0)>0$. By the Morse property (see for example \cite[Theorem III.H.1.7]{Bridson-Haefliger99}), there exists a constant $K$ such that for any $\zeta$ in the Gromov boundary and for any $n$, the intersection of the sphere $S_n$ of radius $n$ in $\Gamma$ with the set of all geodesics from $e$ to $\zeta$ contains at most $K$ elements. Assume that $y_m$ in $\Gamma$ converges to $\xi$ in the 0-Martin boundary. Then, since $|y_m|$ is the first time in which we visit $y_m$, for every $n$ and sufficiently large $m$ we have that,
    \begin{equation*}
        P^{|y_m|}(e, y_m) = \sum_{x\in S_n\cap \mathcal{P}(\xi)}P^n(e, x)P^{|y_m| - n}(x, y_m).
    \end{equation*}
    Hence, after dividing by $P^{|y_m|}(e, y_m)$ and letting $m$ tend to infinity we get that,
    \begin{equation*}
        1 = \sum_{x\in S_n\cap \mathcal{P}(\xi)}P^n(e, x)K(x, \xi \,|\, 0).
    \end{equation*}
    Now, since $P^n(e,x)\leq R^{-n}$ for every $x \in \Gamma$, we get that there exists at least one point $x \in S_n\cap \mathcal{P}(\xi)$ such that $K(x, \xi \,|\, 0)\geq R^n/K$. In particular, for $\xi\in \partial_{M, 0}\Gamma$,
    \begin{equation*}
        \sum_{x\in S_n\cap \mathcal{P}(\xi)}K(x, \xi \,|\, 0)\geq \frac{R^n}{K}.
    \end{equation*}

    Now, we fix $\xi$ and we let $\eta$ be such that $\psi_\mu(\xi)\neq \psi_\mu(\eta)$. There exists an open neighborhood $U$ of $\psi_\mu(\eta)$ with $\psi_\mu(\xi) \notin U$ and $N \in \mathbb{N}$ so that for any point $\zeta\in U$ in the Gromov boundary, any geodesic from $e$ to $\zeta$ is disjoint from any geodesic from $e$ to $\psi_\mu(\xi)$ outside of the ball $B_{N}$ around $e$. We let $\mathcal{P}(U)$ be the set of $x \in \Gamma$ such that $K(x, \tau \, | \, 0) > 0$ for some $\tau \in \psi_\mu^{-1}(U)$. By what precedes, $\mathcal{P}(U)$ contains $\mathcal{P}(\tau)$ for every $\tau \in \psi_\mu^{-1}(U)$ and for $n>N$, we have that $\mathcal{P}(U)\cap S_n$ and $\mathcal{P}(\xi)\cap S_n$ are disjoint. In particular, we find that for $n> N$
    \begin{equation*}
        \sum_{x\in S_n\cap \mathcal{P}(U)}K(x, \xi \,|\,0) = 0,
    \end{equation*}
    while for any $\tau\in \psi_\mu^{-1}(U)$,
    \begin{equation*}
        \sum_{x\in S_n\cap \mathcal{P}(U)}K(x, \tau \,|\, 0)\geq \frac{R^n}{K}.
    \end{equation*}
    Consequently,
    \begin{align*}
        0 = & \sum_{x\in S_n\cap \mathcal{P}(U)}K(x, \xi \,|\,0) =\ \sum_{x\in S_n\cap \mathcal{P}(U)}\int_{\partial_{M,0}\Gamma} K(x, \tau \,|\, 0)d\nu_{\xi}(\tau) \\
        & \ \geq \sum_{x\in S_n\cap \mathcal{P}(U)}\int_{\psi_\mu^{-1}(U)} K(x, \tau \,|\, 0)d\nu_{\xi}(\tau) \geq \frac{R^n}{K} \cdot \nu_{\xi}(\psi_\mu^{-1}(U)).
    \end{align*}
    We deduce that $\nu_{\xi}(\psi_\mu^{-1}(U)) = 0$. Since $\psi_\mu^{-1}(U)$ is open and contains $\eta$, this proves that $\eta$ is not in the support of $\nu_{\xi}$.
\end{proof}

\begin{corollary}

    There is at least one minimal point in the preimage $\psi_\mu^{-1}(\zeta)$ for every $\zeta$ in the Gromov boundary.
\end{corollary}
For $\lambda > 0$, the $\lambda$-Martin boundary is equivariantly homeomorphic to Gromov boundary and consists of \emph{minimal} $\lambda$-harmonic functions \cite{Ancona, Lalley-Gouezel13, GouezelLLT}. The next example shows that the same does not generally hold for the $0$-Martin boundary.

\begin{example}
\label{ex:free:0-martin}

    Let $\F_{2}$ be the free group with two generators denoted by $a$ and $b$, and let $\Z/2\Z = \{0,1\}$ be the finite group of two elements. We consider a finitely generated hyperbolic group $\Gamma = \mathbb F_{2} \times \Z/2\Z$ together with the lazy simple random walk with a measure whose support is
    \begin{equation*}
        S = \{(a, 0), (b, 0), (a^{-1}, 0), (b^{-1}, 0), (a, 1), (a^{-1}, 1)\}.
    \end{equation*}
    In other words, we endow $\Gamma$ with an admissible finitely supported lazy symmetric probability measure $\mu$ by setting
    \begin{equation*}
        \mu = \alpha \delta_{(e, 0)} + \frac{1 - \alpha}{6} \big(\delta_{(a, 0)} + \delta_{(a^{-1}, 0)} + \delta_{(b, 0)} + \delta_{(b^{-1}, 0)} + \delta_{(a, 1)} + \delta_{(a^{-1}, 1)}\big).
    \end{equation*}
    for some $0 < \alpha < 1$. We denote by $|\cdot|$ the word distance induced by $\mu$, which is just the word distance associated with the generating set $S$.
\end{example}

\begin{proposition}
\label{STcounterexample}

    In Example \ref{ex:free:0-martin}, the map $\psi_\mu$ from $0$-Martin boundary onto Gromov boundary is not injective. Moreover, the $0$-Martin boundary contains non-minimal $\infty$-harmonic functions.
\end{proposition}

\begin{proof}
    We consider the three sequences of $\Gamma$ defined by $y_n^{(0)} = (ab^n, 0)$, $y_n^{(1)} = (ab^n, 1)$ and $y_n^{(2)} = (ab^na, 1)$. First, we compute the geodesics from $(e, 0)$ to $y_n^{(i)}$. There is a unique geodesic from $(e, 0)$ to $y_n^{(0)}$ which is given by
    \begin{equation*}
        \{(e, 0), (a, 0), (ab, 0), \dots, (ab^n, 0)\}.
    \end{equation*}
    In particular, $|y_n^{(0)}| = n + 1$. There is also a unique geodesic from $(e, 0)$ to $y_n^{(1)}$ which is given by
    \begin{equation*}
        \{(e, 0), (a, 1), (ab, 1), \dots, (ab^n, 1)\}.
    \end{equation*}
    In particular, $|y_n^{(1)}| = n + 1$. Indeed, to go from $(e, 0)$ to $(ab^n, 1)$, one needs to use the element $(a, 1)$ or the element $(a^{-1}, 1)$ at some point. However, if this element is not the first increment or is not $(a, 1)$, then one needs to simplify the up-coming letter $a$ or $a^{-1}$ on the factor $\mathbb F_{2}$, which yields a path of length at least $n + 2$. Similarly, to go from $(e, 0)$ to $y_n^{(2)}$, one also needs to use the element $(a, 1)$ at some point. This time, this has to be either the first of the last increment of the geodesic. Thus, there are exactly two geodesics from $(e, 0)$ to $y_n^{(2)}$ which are given by
    \begin{equation*}
        \{(e, 0), (a, 0), (ab, 0), \dots, (ab^n, 0), (ab^na, 1)\}
    \end{equation*}
    and
    \begin{equation*}
        \{(e, 0), (a, 1), (ab, 1), \dots, (ab^n, 1), (ab^na, 1)\}.
    \end{equation*}
    In particular, $|y_n^{(2)}| = n + 2$.

    We start by showing that all elements $y_n^{(0)}$ are aligned on a geodesic. Consequently, this sequence converges to a point in the Gromov boundary of $\Gamma$ that we denote by $ab^\infty$. Second, note that
    \begin{equation*}
        y_n^{(2)} = y_n^{(0)}\cdot (a, 1)
    \end{equation*}
    and
    \begin{equation*}
        y_n^{(1)} = y_n^{(0)}\cdot (a, 1) \cdot (a^{-1}, 0) = y_n^{(2)}\cdot (a^{-1}, 0).
    \end{equation*}
    In particular for every $n$, the three points $y_n^{(i)}$ stay within a distance at most $2$ of each other. Thus, the three sequences $y_n^{(i)}$ converge to the same point $ab^\infty$ in the Gromov boundary.

    Next, we prove that these three sequences converge to the 0-Martin boundary and compute their respective limit according to the $0$-Martin kernel. We set
    \begin{equation*}
        \mathcal{H}_0 = \{(e, 0)\}\cup \{(ab^n, 0), n\geq 0\}
    \end{equation*}
    and
    \begin{equation*}
        \mathcal{H}_1 = \{(e, 0)\}\cup\{(ab^n, 1), n\geq 0\}.
    \end{equation*}

    Let $x\in \Gamma$ and set $m = |x|$. If $x\notin\mathcal{H}_0$, then $x$ is not on a geodesic from $e$ to $y_n^{(0)}$ for any $n$. Thus, $P^{n + 1 - m}(x, y_n^{(0)}) = 0$. On the contrary, assume that $x = (ab^{m - 1}, 0)$, with $m\geq 1$ or that $x = (e, 0)$ with $m = 0$. Then, for every $n\geq m - 1$, there is a unique trajectory for the random walk of length $n + 1 - m$ from $x$ to $y_n^{(0)}$ and
    \begin{equation*}
        P^{n - m + 1}(x, y_n^{(0)}) = \left (\frac{1 - \alpha}{6}\right)^{n - m + 1},
    \end{equation*}
    hence
    \begin{equation*}
        \frac{P^{n - m + 1}(x, y_n^{(0)})}{P^{n + 1}((e, 0), y_n^{(0)})}\underset{n\to \infty}{\longrightarrow}\left (\frac{1 - \alpha}{6}\right)^{-m}.
    \end{equation*}
    Consequently, $y_n^{(0)}$ converges to a point in the $0$-Martin boundary and the limit $h_0$ of the $0$-Martin kernel is given by
    \begin{equation*}
        h_0(x) = 
        \begin{dcases}
            0, & \text{if } x\notin\mathcal{H}_{0}, \\
            \left(\frac{1 - \alpha}{6}\right)^{-|x|}, & \text{otherwise.}
        \end{dcases}
    \end{equation*}
    Next, we consider $y_n^{(1)}$. Let $x\in \Gamma$ and set $m = |x|$. Since there also exists a unique geodesic from $(e, 0)$ to $y_n^{(1)}$, we similarly find that $P^{n - m + 1}(x, y_n^{(1)}) = 0$ if $x\notin \mathcal{H}_1$ and that
    \begin{equation*}
        \frac{P^{n - m + 1}(x, y_n^{(1)})}{P^{n + 1}((e, 0), y_n^{(1)})}\underset{n\to\infty}{\longrightarrow}\left(\frac{1 - \alpha}{6}\right)^{-m}
    \end{equation*}
    otherwise.
    Thus, $y_n^{(1)}$ converges to a point in the 0-Martin boundary and the limit $h_1$ of the $0$-Martin kernel is given by
    \begin{equation*}
        h_1(x) =
        \begin{dcases}
            0, & \text{if } x\notin\mathcal{H}_{1}, \\
            \left(\frac{1 - \alpha}{6}\right)^{-|x|}, & \text{otherwise.}
        \end{dcases}
    \end{equation*}
    
    Before we compute the limit of $y_n^{(2)}$, we note that $h_0\neq h_1$ and, in particular, $y_n^{(0)}$ and $y_n^{(1)}$ do not converge to the same point in the 0-Martin boundary, while they both converge to $ab^\infty$ in the Gromov boundary. Together with Proposition~\ref{p:zeroMartintoGromov}, this already shows that the $0$-Martin compactification is bigger than the Gromov compactification.

    \medskip
    Next, we now show that $y_n^{(2)}$ also converges and that the limit of the $0$-Martin kernel is not a minimal $\infty$-harmonic function. Set
    \begin{equation*}
        \gamma_n^{(0)} = \{(e, 0), (a, 0), (ab, 0), \dots, (ab^n, 0), (ab^na, 1)\}
    \end{equation*}
    and
    \begin{equation*}
        \gamma_n^{(1)} = \{(e, 0), (a, 1), (ab, 1), \dots, (ab^n, 1), (ab^na, 1)\}.
    \end{equation*}
    These are the only two geodesics from $(e, 0)$ to $y_n^{(2)}$. Let $x\in\Gamma$, $x\neq (e, 0)$, and set $m = |x|$, so that $m\geq 1$. If $x\notin\mathcal{H}_0\cup\mathcal{H}_1$, then for $n \geq m - 1$ we have that $x$ is neither on $\gamma_n^{(0)}$, nor on $\gamma_n^{(1)}$, and so $P^{n - m + 2}(x, y_n^{(2)}) = 0$. Assume now that $x = (ab^{m - 1}, 0)$. Then, $x$ is on $\gamma_n^{(0)}$ for $n \geq m - 1$. Moreover, the two geodesics $\gamma_n^{(0)}$ and $\gamma_n^{(1)}$ only intersect at $(e, 0)$ and $y_n^{(2)}$. Since $x\neq (e, 0)$, we find that there is exactly one trajectory for the random walk of length $n + 2 - m$ from $x$ to $y_n^{(2)}$ and that
    \begin{equation*}
        P^{n - m + 2}(x, y_n^{(2)}) = \left (\frac{1 - \alpha}{6}\right)^{n - m + 2}.
    \end{equation*}
    Similarly, if $x = (ab^{m - 1}, 1)$, then
    \begin{equation*}
        P^{n - m + 2}(x, y_n^{(2)}) = \left (\frac{1 - \alpha}{6}\right)^{n - m + 2}.
    \end{equation*}
    Finally, since there are exactly $2$ trajectories of length $n + 2$ from $(e, 0)$ to $y_n^{(2)}$ we get
    \begin{equation*}
        P^{n + 2}((e, 0), y_n^{(2)}) = 2 \left(\frac{1 - \alpha}{6}\right)^{n + 2}.
    \end{equation*}
    We deduce that if $x \neq (e, 0)\in \mathcal{H}_0\cup\mathcal{H}_1$, then
    \begin{equation*}
        \frac{P^{n - m + 2}(x, y_n^{(1)})}{P^{n + 2}((e, 0), y_n^{(1)})}\underset{n\to\infty}{\longrightarrow}\frac{1}{2}\left(\frac{1 - \alpha}{6}\right)^{-m}.
    \end{equation*}
    Thus, $y_n^{(2)}$ also converges to a point in the $0$-Martin boundary, and the limit $h_2$ of the $0$-Martin kernel is given by
    \begin{equation*}
        h_2(x) =
        \begin{dcases}
            0, & \text{if } x\notin\mathcal{H}_0\cup \mathcal{H}_1, \\
            1, & \text{if } x = (e, 0), \\
            \left(\frac{1 - \alpha}{6}\right)^{-|x|}, & \text{otherwise,}
        \end{dcases}
    \end{equation*}
    and see that
    \begin{equation*}
        h_2 = \frac{1}{2}\left(h_0 + h_1\right).
    \end{equation*}
    Therefore, $h_2$ is not minimal.
\end{proof}

\section{Minimal space-time boundary}
\label{s:minimalspacetime}

In this section, for an admissible lazy random walk on $\Gamma$, we will identify the minimal space-time Martin boundary with the whole disjoint union of minimal $\lambda$-Martin boundaries for $\lambda \in [0,R]$. We begin with an key definition.

\begin{definition}

    A function $f\colon\ST\to \mathbb R$ is called \textit{space-time harmonic} if it satisfies $P_{\ST}f = f$, where $P_{ST}$ is the Markov kernel associated with the space-time Markov chain $(Y_n)_{n\in \N}$. A non-negative space-time harmonic function $f$ is said to be \textit{normalized} if $f((e, 0)) = 1$. 
    Furthermore, a non-negative space-time harmonic function $f$ is called \textit{minimal} if for every other non-negative space-time harmonic function $g$ satisfying $c \cdot g\leq f$ for some constant $c > 0$, we must have that $f = \alpha g$ for some $\alpha>0$. 
\end{definition}

Notice that the Markov chain $(Y_n)_{n\in \N}$ is far from being irreducible on $\ST$. Thus, the classical proof of the minimum principle (see \cite[(1.15)]{Woessbook}) for super-harmonic functions does not apply to the chain $(Y_n)_{n\in \N}$. So, nothing apriori prevents a non-zero positive space-time harmonic function from vanishing on some points in the space-time set. In fact, we will construct non-zero positive space-time harmonic functions that do vanish on some points of the space-time set below.

In \cite[Theorem~3.1]{LS63}, for a wide class of Markov chains, strictly positive minimal space-time harmonic functions for an \emph{$\alpha$-modified} version of the space-time Markov chain were characterized. These were shown to have the form $f_{c, h}\colon (x, m) \mapsto c^m h(x)$, where $h$ is a $\frac{1 - \alpha}{1 - \alpha c}$-harmonic function for the original Markov chain with $\frac{1}{\alpha} > c \geq 0$. Thus, if the original Markov chain is irreducible, then every non-zero $f_{c,h}$ is strictly positive.

Since modified space-time Markov chains and their associated harmonic functions are defined on a the Cartesian product state space $\Gamma \times \mathbb{N}$, which is a bigger state space than $\ST$, in general it is unclear what the precise relationship is between these different notions of minimal space-time harmonic functions (see also \cite{Molchanov} for results on space-time Markov chains on Cartesian products). However, a description of minimal space-time harmonic functions for the space-time Markov chain on $\ST$ was obtained in \cite[Theorem~3.2]{LS63} for certain nearest neighbor Markov chains on the non-negative integers, and, just as it is shown there, we will see below that for lazy random walks there always exist non-negative non-zero space-time harmonic functions that still vanish on some points of the space-time set.

For $f$ space-time harmonic, we have that
\begin{equation*}
    f((x, m)) = \sum_{y\in \Gamma}\mu(x^{-1}y)f((y, m + 1)) = \sum_{y\in \Gamma}P^{n}(x, y)f((y, m + n))
\end{equation*}
for $(x, m)\in \ST$ and $n\geq 0$, so we find that
\begin{equation}
\label{eq:space-time:harmonic}
    f((x, m))\geq P^{n}(e, e)f((x, m + n)).
\end{equation}
For each $n\geq 0$, denote by $g_{n} \colon\ST\to\R$ the function defined by 
\begin{equation*}
    g_n((x, m)) = f((x, m + n)).
\end{equation*}
Notice that it is well defined, since for every $(x, m)\in \ST$ and for every $n\geq 0$, we have that $(x, m + n) \in \ST$ as well since we assumed that $\mu$ is lazy.

\begin{lemma}
The function $g_{n}$ is space-time harmonic for each $n\geq 0$.
\end{lemma}

\begin{proof}

    For $(x, m)\in \ST$, we have
    \begin{equation*}
        \begin{split}
            \sum_{y\in \Gamma}\mu(x^{-1}y)g_n((y, m + 1)) & = \sum_{y\in \Gamma}\mu(x^{-1}y)f((y, m + n + 1)) \\
            & = f((x, m + n)) = g_{n}((x, m)),
        \end{split}
    \end{equation*}
    which is the desired equality.
\end{proof}

Now, assume that $f$ is a normalized minimal space-time harmonic function. Using equation \eqref{eq:space-time:harmonic}, we get $g_{n}\leq c_{n}f$, where $c_{n} = 1/P^{n}(e, e)$. Thus, there exists a constant $C_{n}$ such that $g_{n} = C_{n}f$. Evaluating it at $(e, 0)$, we obtain that $C_n = f((e, n))$. Hence, for every $(x, m)\in\ST$ and every $n\geq 0$,
\begin{equation}
\label{eq:minimal}
    f((x, m + n)) = f((e, n))f((x, m)).
\end{equation}

We will now use the word semi-norm $|\cdot|$ introduced in Section~\ref{s:zeroMartin}. By induction on equation \eqref{eq:minimal}, for every $x\in \Gamma$ and every $m\geq |x|$,
\begin{equation}
\label{equation2minimal}
    f((x, m)) = \lambda^{m - |x|}f((x, |x|)),
\end{equation}
where $\lambda = f((e, 1))$. Now suppose that $\lambda >0$ and denote by $g \colon\Gamma\to \R$ the function given by 
$$
g(x) = \lambda^{-|x|}f((x,|x|)),
$$ 
which is clearly well-defined since $\mu$ is lazy.

\begin{lemma} \label{lemma:minimal:lambda-harmonic}
The function $g$ is normalized minimal $\lambda^{-1}$-harmonic for $\lambda > 0$.
\end{lemma}

\begin{proof}
We first prove that $g$ is $\lambda^{-1}$ harmonic. We need to prove that for every $x\in \Gamma$,
    \begin{equation*}
        \lambda^{-1}g(x) = \sum_{y\in \Gamma}\mu(x^{-1}y)g(y) = \sum_{y\in \Gamma}\mu(x^{-1}y)\lambda^{-|y|}f((y, |y|)).
    \end{equation*}
Note that if $\mu(x^{-1}y) > 0$, we necessarily have that $|x| + 1\geq |y|$. By equation \eqref{equation2minimal}, we get
    \begin{equation*}
        \lambda^{-|y|}f((y, |y|)) = \lambda^{-1 - |x|}f((y, |x| + 1)).
    \end{equation*}
Since $f$ is space-time harmonic, we obtain
    \begin{equation*}
        \begin{split}
            \sum_{y\in \Gamma}\mu(x^{-1}y)g(y) & = \lambda^{-1 - |x|}\sum_{y\in \Gamma}\mu(x^{-1}y)f((y, |x| + 1)) \\
            & = \lambda^{-1 - |x|}f((x, |x|)) = \lambda^{-1}g(x).
        \end{split}
    \end{equation*}
Finally, as $f$ is normalized, it follows that $g(e) = 1$.

Next, we prove that $g$ is minimal. Let $h$ be a normalized non-negative $\lambda^{-1}$-harmonic function on $\Gamma$ such that $h\leq cg$ for some constant $c > 0$.
Define for $(x, m)\in \ST$, the function 
    \begin{equation*}
        \Tilde{h}\colon\ST\to \R, (x, m)\mapsto \lambda^mh(x).
    \end{equation*}
We have $\Tilde{h}((x, m))\leq c\lambda^mg(x) = c\lambda^{m - |x|}f(x, |x|)$ for every $(x, m)\in \ST$. Using equation \eqref{equation2minimal}, we obtain $\Tilde{h}\leq c f$. Straightforward computations show that $\Tilde{h}$ is a non-negative space-time harmonic function. Hence, by minimality of $f$, there exists $C > 0$ such that $\Tilde{h} = Cf$. That is, for every $x\in \Gamma$,
    \begin{equation*}
        \lambda^{|x|}h(x) = \Tilde{h}(x, |x|) = Cf((x, |x|)) = C\lambda^{|x|}g(x).
    \end{equation*}
Thus, we have $h = Cg$, i.e. the function $g$ is indeed minimal.
\end{proof}

We are now left dealing with the case of minimal space-time harmonic functions $f$ such that $\lambda = f((e, 1)) = 0$. According to equation \eqref{eq:minimal}, we have that $f((x, m)) = 0$ for all $(x, m)\in \ST$ with $m > |x|$. Thus, the function $f$ can be positive only at $(x, |x|)$ for $x\in \Gamma$. Denote by $g \colon\Gamma\to \R$ the function given by 
$$
g(x) =f((x, |x|)),
$$ 
which is again well-defined when $\mu$ is lazy.

\begin{lemma}
\label{lemma:minimal:infinity-harmonic}

    The function $g$ is minimal $\infty$-harmonic for $\lambda = 0$.
\end{lemma}

\begin{proof}

    Since $f((x, m)) = 0$ whenever $m > |x|$, we get
    \begin{equation*}
        \begin{split}
            \sum_{y\in \Gamma}\bar\mu_x(y)f((y, |y|)) & = \sum_{y,\, |y| = |x| + 1}\mu(x^{-1}y)f((y, |x| + 1)) \\
            & = \sum_{(y, n)\in \ST}\mu(x^{-1}y)f((y, |x| + 1)) \\
            & = f((x, |x|)),
        \end{split}
    \end{equation*}
    where the last equality follows from space-time harmonicity of $f$. This proves that $g$ is non-negative $\infty$-harmonic.
    
    Now, let $h$ be a normalized non-negative $\infty$-harmonic function such that $h\leq cg$ for some $c > 0$. Define $\Tilde{h}((x, m)) = h(x)$ if $m = |x|$ and $0$ otherwise. As before, we get that $\Tilde{h}$ is non-negative space-time harmonic satisfying $\Tilde{h}\leq cf$. Using minimality, we thus obtain that $\Tilde{h} = cf$, so that $h = cg$.
\end{proof}

\begin{theorem}
\label{thm:minimal:harmonic:form}

    Every minimal non-negative space-time harmonic function $f$ can be written for $(x, m)\in\ST$ as either
    \begin{equation*}
        f_{\lambda, \xi}((x, m)) = \lambda^{m}K(x, \xi \,|\, \lambda)
    \end{equation*}
    for some $\xi\in \partial_{M, \lambda}^m\Gamma$ when $f((e, 1)) = \lambda > 0$, or as
    \begin{equation*}
        f_{0, \xi}((x, m)) = \mathbf 1_{m = |x|}K(x, \xi \,|\, 0)    
    \end{equation*}
    for some $\xi\in \partial^{m}_{M, 0}\Gamma$ when $f((e,1)) = 0$. 
    
    In particular, a minimal space-time harmonic function is strictly positive if and only if $\lambda > 0$.
\end{theorem}

\begin{proof}

    First, assume that $\lambda =f((e, 1)) > 0$. Then by Lemma \ref{lemma:minimal:lambda-harmonic}, the associated function $g \colon \Gamma \rightarrow \mathbb{R}$ given by $g(x) = \lambda^{-|x|}f((x,|x|))$ is a minimal $\lambda^{-1}$-harmonic function. Using \cite[Lemma~7.2]{Woessbook}, we may write $g = K(\cdot, \xi \,|\, \lambda)$ for some point $\xi\in\partial_{M, \lambda}^m\Gamma$. Combining this with equation \eqref{equation2minimal}, we get that $f((x, m)) = \lambda^{m}K(x, \xi \,|\, \lambda)$.
    
    Similarly, if $\lambda = f((e, 1)) = 0$, then by Lemma \ref{lemma:minimal:infinity-harmonic}, the associated function $g \colon \Gamma \rightarrow \mathbb{R}$ given by $g(x) = f(x, |x|)$ is a minimal $\infty$-harmonic function. By Theorem \ref{thm:Martin-Poisson:sub-Markov}, we have that $g = K(\cdot, \xi \,|\, 0)$ for some $\xi \in \partial_{M,0}^m \Gamma$. Combining this with equation \eqref{equation2minimal}, we get that $f((x, m)) = \mathbf 1_{m = |x|}K(x, \xi \,|\, 0)$.
\end{proof}

In particular, a minimal space-time harmonic function $f$ is completely determined by $\lambda$ and $\xi$, so the minimal space-time Martin boundary is seen to be naturally identifiable with a subset of the disjoint union of $\lambda$-Martin boundaries, i.e. we have the following injective map
\begin{equation*}
    \iota\colon \partial^{m}_{\ST}\Gamma \xhookrightarrow{} \bigsqcup_{\lambda\in [0, R]}\partial_{M, \lambda}^m\Gamma.
\end{equation*}

Recall that $\tilde K (x, \xi \,|\, \lambda) = \lambda^{|x|}K(x, \xi \,|\, \lambda)$ for $\lambda > 0$ and $\xi \in \partial_{M,\lambda}^m\Gamma$, and $\tilde K(x, \xi \,|\, 0) = K(x, \xi \,|\, 0)$ for $\xi \in \partial_{M,0}^m\Gamma$. Now, we endow the right-hand side with the natural topology of point-wise convergence in which a sequence $(\lambda_{n}, \xi_{n})$ converges to the point $(\lambda, \xi)$ if and only if $\lambda_{n}\to\lambda$ and $\widetilde{K}(x, \xi_{n} \,|\, \lambda_{n}) \to \widetilde{K}(x, \xi \,|\, \lambda)$. We now show that the induced topology on the space-time Martin boundary coincides with the space-time topology ensuring that $\iota$ is indeed continuous.

\begin{corollary}
\label{prop:minimal:convergence}
The map $\iota$ is continuous.
\end{corollary}

\begin{proof}
By Proposition \ref{p:all-bnry-harm}, we have that $K_{\ST}((x,m),\eta)$ is space-time harmonic for every $\eta \in \partial_{\ST} \Gamma$. Now, given $\eta \in \partial_{\ST}^m \Gamma$, we know that there exists $(\lambda,\xi) \in \sqcup_{\lambda\in [0, R]}\partial_{M, \lambda}^m\Gamma$ such that $K_{\ST}((x,m),\eta) = f_{\lambda,\xi}((x,m))$. In particular, whenever $K_{\ST}((e, 1),\eta) = \lambda > 0$, we have $K(x,\xi \, | \, \lambda) = \lambda^{-|x|}K_{\ST}((x, |x|),\eta)$, and when $ K_{\ST}((e, 1), \eta) = \lambda = 0$, we have $K(x, \xi \, | \, 0) = K_{\ST}((x, |x|),\eta)$.

Suppose we have a sequence $\eta_n \in \partial^{m}_{\ST} \Gamma$ converging to $\eta \in \partial^{m}_{\ST} \Gamma$, and denote by $(\xi_n,\lambda_n)$ the element corresponding to $\eta_n$ and by $(\xi,\lambda)$ the element corresponding to $\eta$. Then, using the above formulas, it is clear that $\lambda_n \rightarrow \lambda$, and $\xi_n \rightarrow \xi$.
\end{proof}

We now arrive at the main result of the section, showing that the minimal space-time Martin boundary is exactly the whole disjoint union of minimal $\lambda$-Martin boundaries.

\begin{theorem}
\label{t:minimalspacetime}

    Let $\Gamma$ be a countable discrete group, and let $\mu$ be an admissible finitely supported lazy probability measure on $\Gamma$. Then, the minimal space-time Martin boundary is homeomorphic to the disjoint union of the $\lambda$-minimal Martin boundaries for $\lambda\in [0, R]$ with its natural topology of point-wise convergence.
\end{theorem}

\begin{proof}

    Using Theorem \ref{thm:minimal:harmonic:form}, we are left to show the map $\iota$ is surjective, i.e. that each element of the disjoint union is minimal space-time harmonic. In other words, we need to show that the space of minimal non-negative space-time harmonic functions is comprised exactly of the functions of the form
    \begin{equation*}
        f_{\lambda, \xi}((x, m)) = \lambda^{m}K(x, \xi \,|\, \lambda)
    \end{equation*}
    for $\lambda = f_{\lambda, \xi}((e, 1)) > 0$ and $\xi$ in $\partial_{M, \lambda}^m\Gamma$,
    and
    \begin{equation*}
        f_{0, \xi}((x, m)) = \mathbf{1}_{m = |x|}K(x, \xi \,|\, 0)
    \end{equation*}
    for $f_{0, \xi}((e,1)) = 0$ and $\xi\in \partial^{m}_{M, 0}\Gamma$.
    
    First, note that every function $f_{\lambda, \xi}$ is non-negative space-time harmonic for $\lambda\in[0, R]$. Therefore, there exists a probability measure $\nu$ with full mass on the minimal space-time Martin boundary identified as a subset of the disjoint union of $\lambda$-Martin boundaries such that
    \begin{equation}
    \label{eq:space-time:harmonic:integral}
        f_{\lambda, \xi}((x, m)) = \int_{\bigsqcup_{\lambda\in [0, R]}\partial^{m}_{M, \lambda}\Gamma} f_{r, \zeta}((x, m))d\nu(r, \zeta).
    \end{equation}
    
    Notice that by Corollary \ref{prop:minimal:convergence}, the canonical projection from the minimal space-time Martin boundary to $[0, R]$ defined by
    \begin{equation*}
        \pi\colon \partial^{m}_{\ST}\Gamma\to [0, R], (\lambda, \xi)\mapsto \lambda
    \end{equation*}
    is continuous. Hence, given a probability measure $\nu$ on the minimal space-time Martin boundary, its push-forward $\pi_*\nu$ is a probability measure on $[0, R]$ endowed with the standard Borel $\sigma$-algebra. By definition, the measure $\pi_*\nu$ satisfies 
    \begin{equation}
    \label{e:integralpushforward}
        \int_0^R fd\pi_*\nu = \int_{\bigsqcup_{\lambda\in [0, R]}\partial^{m}_{M, \lambda}\Gamma}f\circ \pi d\nu
    \end{equation}
    for all bounded continuous functions $f$ on $[0, R]$. 
    
    Now, evaluating equation \eqref{eq:space-time:harmonic:integral} at $(e, m)$ and applying \eqref{e:integralpushforward} to the function $f\colon r \mapsto r^m$, we get
    \begin{equation*}
        \lambda^m = \int_{\bigsqcup_{\lambda\in [0, R]}\partial^{m}_{M, \lambda}\Gamma} r^md\nu(r, \zeta) = \int_{0}^{R}r^md\pi_*\nu(r),
    \end{equation*}
    whether $\lambda = 0$ or not. Since a probability measure on $[0, R]$ is determined by its moments (see, e.g. \cite[Theorem~30.1]{Bill86}), we deduce that $\pi_*\nu = \delta_{\lambda}$. This implies that $\nu$ gives full measure to the fiber $\partial_{M, \lambda}^m\Gamma$ and so can be considered as a probability measure on $\partial_{M, \lambda}^m\Gamma$. On the other hand, evaluating equation \eqref{eq:space-time:harmonic:integral} at $(x,|x|)$, we obtain
    \begin{equation*}
        K(x, \xi \,|\, \lambda) = \int_{\partial^{m}_{M, \lambda}\Gamma} K(x, \zeta \,|\, \lambda)d\nu(\zeta).
    \end{equation*}
    Since $\xi$ is a point in the minimal $\lambda$-Martin boundary, we necessarily get that $\nu = \delta_\xi$. Thus, we proved that any $f_{\lambda, \xi}$ cannot be written as a non-trivial integral on the space-time Martin boundary, so it must therefore be a minimal space-time Martin function.
\end{proof}

\begin{remark}
   In view of this proposition, it is natural to ask if the same type of result hold for the whole space-time Martin boundary. It turns out that the space-time Martin boundary often does \textbf{not} coincide with the disjoint union of $\lambda$-Martin boundaries. Indeed, due to Theorem \ref{t:ratiolimitspacetime}, we know that the whole ratio-limit compactification embeds in the space-time Martin boundary for symmetric random walks on torsion-free hyperbolic groups. More precisely, for any finitely supported symmetric random walk on a hyperbolic group $\Gamma$ we have that $R_{\mu}$ is contained in the amenable radical of $\Gamma$. Thus, we see that the top cap of the space-time boundary contains a copy of the non-amenable group $\Gamma/R_\mu$ which is not part of the disjoint union of $\lambda$-Martin boundaries.
\end{remark}

\section{C*-envelopes of tensor algebras arising from random walks}
\label{sec:tensor_algebras}

In this section, we aim to extend the results of the first-named author with Markiewicz on $C^{\ast}$-envelopes of tensor algebras for finite stochastic matrices in \cite{DOMA17} to the context of random walks on countably infinite groups. We will apply our description of the minimal space-time Martin boundary from Theorem \ref{t:minimalspacetime}. First, we discuss some preliminary observations about boundary representations of the Toeplitz algebra $\mathcal{T}(\Gamma, \mu)$ with respect to the tensor algebra $\mathcal{T}^+(\Gamma,\mu)$.

For $z\in\Gamma$, we let $\pi_{z}$ denote the representation
\begin{equation}
    \begin{split}
        \pi_{z}\colon \mathcal{T}(\Gamma, \mu) & \to \mathbb{B}(\mathcal{H}_{z}(\Gamma, \mu)), \\
        T & \mapsto T|_{\mathcal{H}_{z}}
    \end{split}
\end{equation}
of the Toeplitz algebra $\mathcal{T}(\Gamma, \mu)$. Our goal is to show that these representations are boundary for the tensor algebra $\mathcal{T}^+(\Gamma, \mu)$ and use this to identify the $C^{\ast}$-envelope of $\mathcal{T}^+(\Gamma, \mu)$ as the Toeplitz algebra $\mathcal{T}(\Gamma,\mu)$. Recall that we denote 
\begin{equation*}
    \mathcal{H}^{(m)} := \bigoplus_{z\in\Gamma}\ell^{2}(\ST^{(m)}_{z}),
\end{equation*}
where $\ST^{(m)}_{z} := \{y\in\Gamma \,|\, P^{m}(y, z) > 0\}$.

\begin{proposition}

    Let $\Gamma$ be a countable discrete group, and $\mu$ be an admissible finitely supported probability measure on $\Gamma$. Then, $\pi_{z}$ is an irreducible representation for each $z\in\Gamma$.
\end{proposition}

\begin{proof}

    By \cite[Proposition 4.4]{DOAD21} we get for any $z\in\Gamma$ that $\K(\mathcal{H}_{z})\lhd\pi_{z}(\mathcal{T}(\Gamma, \mu))$. Since the image of $\pi_{z}$ contains a copy of compact operators, it follows that $\pi_{z}$ is an irreducible representation.
\end{proof}

In what follows, for $\ell \in\N$ we denote by $Q^{(\ell)}$ the orthogonal projection from $\mathcal{H}$ onto $\mathcal{H}^{(\ell)}$, and for $x\in\Gamma$ we define the projection $Q^{(\ell)}_{x} := S^{(0)}_{x, x}Q^{(n)} = Q^{(\ell)}S^{(0)}_{x, x}$ onto the standard basis vectors $\{ \, e^{(\ell)}_{x,z} \, | \, z\in \Gamma, \, P^{\ell}(x,z)>0 \, \}$. By \cite[Proposition 4.4]{DOAD21} we know that $Q^{(\ell)}_x \in \I(\Gamma,\mu)$ for all $\ell \in \N$ and $x\in \Gamma$.

\begin{proposition}

    Let $\Gamma$ be a countable discrete group, and $\mu$ be a admissible finitely supported probability measure on $\Gamma$. Then, the representations $\pi_{z}$ and $\pi_{z'}$ are not unitarily equivalent for distinct $z, z'\in\Gamma$.
\end{proposition}

\begin{proof}
Since $Q^{(0)}_{z}\in\mathcal{T}(\Gamma, \mu)$, the rest follows verbatim as in \cite[Proposition 3.4]{DOMA17}, where the assumption of finiteness of $P$ becomes redundant.
\end{proof}

For $n, m\in\N$ and $x, y$ such that $P^{n}(x, y) > 0$, denote the operators
\begin{equation}
\label{eq:reduced_shifts}
    \begin{split}
        T^{(n)}_{x, y}\colon \mathcal{H}(\Gamma, \mu) & \to \mathcal{H}(\Gamma, \mu), \\
        e^{(m)}_{y', z} & \mapsto \delta_{y, y'}\sqrt{\dfrac{P^{m}(y, z)}{P^{n + m}(x, z)}}e^{(n + m)}_{x, z},
    \end{split}
\end{equation}
where the adjoints are given by
\begin{equation*}
    \begin{split}
        (T^{(n)}_{x, y})^{\ast}(e^{(n + m)}_{x', z}) = \delta_{x, x'}\sqrt{\dfrac{P^{m}(y, z)}{P^{n + m}(x, z)}}e^{(m)}_{y, z}.
    \end{split}
\end{equation*}
Since $T^{(n)}_{x, y} = \dfrac{1}{\sqrt{P^{n}(x, y)}}S^{(n)}_{x, y}$, each $T^{(n)}_{x, y}\in\mathcal{T}(\Gamma, \mu)$ and $T^{(n)}_{x, y}(\mathcal{H}^{(m)}_{z})\subset \mathcal{H}^{(n + m)}_{z}$.

\begin{lemma}
\label{lemma:space_time:inequality}
Let $\Gamma$ be a countable discrete group, and $\mu$ be an admissible finitely supported lazy probability measure on $\Gamma$. Then, there exists $n\in \mathbb{N}$ such that the following strict inequality holds
    \begin{equation*}
        \sup_{(z, m)\in\ST }K_{\ST}((e, n), (z, m)) < \dfrac{1}{P^n(e, e)}.  
    \end{equation*}
\end{lemma}

\begin{proof}

    We consider the space-time set $\ST$ as a dense subset of the space-time compactification $\Delta_{\ST}\Gamma$. Since $\Delta_{\ST}\Gamma$ is compact, and the space-time Martin function $(z,m) \mapsto K_{\ST}((e, n), (z, m))$ extends continuously to $\Delta_{\ST}\Gamma$, we know that that the value of the supremum in the lemma can be attained either on $\ST$ or on the space-time boundary $\partial_{\ST}\Gamma$. 

    First, suppose that the supremum is attained at a point in the boundary $\gamma\in\partial_{\ST}\Gamma$ for arbitrarily large $n\in \mathbb{N}$. Since the space-time Markov chain has finite range, by a combination of Proposition \ref{p:all-bnry-harm} and Theorem \ref{thm:Martin-Poisson:sub-Markov} for transient accessible Markov chains, it follows that
    \begin{equation*}
        K_{\ST}((e, n), \gamma) = \int_{\partial^{m}_{\ST}\Gamma}K_{\ST}((e, n), \zeta)d\nu(\zeta)
        \leq \sup_{\eta\in\partial^{m}_{\ST}\Gamma}K_{\ST}((e, n), \eta)
    \end{equation*}
    for a probability measure $\nu$ on $\partial^{m}_{\ST}\Gamma$. By the identification of the minimal space-time boundary in Theorem \ref{t:minimalspacetime}, we get
    \begin{equation*}
        K_{\ST}((e, n), \gamma) \leq \sup_{(\lambda, \xi)\in\sqcup_{\lambda\in[0, R]}\partial^{m}_{M, \lambda}\Gamma} f_{\lambda,\xi}((e, n)).
    \end{equation*}
    Thus, either $\lambda = 0$ and $\xi\in\partial^{m}_{M, 0}\Gamma$, implying that
    \begin{equation*}
        f_{0, \xi}((e, n)) =  \chi_{|e| = n}K_{0}(e, \xi) = 0 < \frac{1}{P^n(e,e)},
    \end{equation*}
    or $\lambda\in(0,R]$ and $\xi\in\partial^{m}_{M, \lambda}\Gamma$, in which case
    \begin{equation*}
        f_{\lambda ,\xi}((e, n)) = \lambda^n K_{\lambda}(e, \xi) = \lambda^n \leq R^n,
    \end{equation*}
    This reduces our inequality to showing that
    \begin{equation*}
        R^{n}P^{n}(e, e) < 1,
    \end{equation*}
    which follows for sufficiently large $n\in\N$ by \cite[Theorem 7.8]{Woessbook}.

    Second, suppose that for all but finitely many $n\in \mathbb{N}$, the suprema for the functions $(z,m)\mapsto K_{\ST}((e,n),(z,m))$ are attained on $\ST$. In this case, for the value $K_{\ST}((e, n), (z_{0}, m_{0}))$ to be positive, it must be the case that $m_0 \geq n$. If $m_0 = n$ we necessarily have that $z_0 = e$, so that $K_{\ST}((e,n),(z_0,m_0)) = 1 < \frac{1}{P^n(e,e)}$. 
    
    Thus, we may assume that $m_0 > n$. In this case, we necessarily have that $K_{\ST}((e, n), (z_0, m_0)) = P^{m_{0} - n}(e, z_{0})/P^{m_{0}}(e, z_{0})$. In particular, the numerator is positive. Due to the fact that the generating measure $\mu$ is admissible and lazy, we know that $0<P^n(e,e)<1$. In this case, we observe that we necessarily have
    \begin{equation*}
        P^n(e, e)P^{m_{0}-n}(e, z_{0}) < P^{m_{0}}(e, z_{0}),
    \end{equation*}
    so that
    \begin{equation*}
        K_{\ST}((e, n), (z_{0}, m_{0})) = \dfrac{P^{m_{0} - n}(e, z_{0})}{P^{m_{0}}(e, z_{0})} < \dfrac{1}{P^n(e, e)}.
    \end{equation*}

    Thus, we may choose any non-zero $n \in \mathbb{N}$ for which the supremum for the function $(z,m)\mapsto K_{\ST}((e,n),(z,m))$ is attained on $\ST$, and we get our desired strict inequality.
\end{proof}

In what follows, we will employ the natural $\Gamma$ action on $\mathcal{T}(\Gamma,\mu)$. Observe that for every $g\in \Gamma$ the map $\ST_{z}\to\ST_{g^{-1}z}$ given by $(y, m) \mapsto (g^{-1}y, m)$ induces unitary operators $U_{g,z}\colon \ell^2(\ST_z) \rightarrow \ell^2(\ST_{g^{-1}z})$ given by $e^{(m)}_{x, z} \mapsto e^{(m)}_{g^{-1}x, g^{-1}z}$, which then yields a unitary
\begin{equation*}
    U_g := \SOT-\sum_{z\in \Gamma} U_{g, z} \colon \mathcal{H}(\Gamma, \mu) \to \mathcal{H}(\Gamma, \mu).
\end{equation*}
Since $U_g(S^{(n)}_{x, y})U_g^{-1} = S^{(n)}_{g^{-1}x, g^{-1}y}$, we obtain an action $\Ad_{U}\colon \Gamma\curvearrowright \mathcal{T}(\Gamma,\mu)$ via $\ast$-automorphisms given by $(\Ad_{U})_g := \Ad_{U_g}$ where $\Ad_{U_g}(T) = U_{g}TU_{g^{-1}}$. It is clear that $\Ad_U$ leaves $\mathcal{T}^+(\Gamma, \mu)$ invariant, and the description of $\I(\Gamma,\mu)$ as the ideal generated by $\{Q_z^{(0)}\}_{z\in \Gamma}$ in \cite[Proposition 4.4]{DOAD21} shows that $\Ad_U$ also leaves $\I(\Gamma,\mu) =\bigoplus_{z\in \Gamma} \mathbb{K}(\H_z)$ invariant. The action $\Ad_U$ induces an action $\widehat{\Ad_{U}}\colon \Gamma \curvearrowright \sp({\mathcal{T}(\Gamma, \mu)})$ on the irreducible representation spectrum of $\mathcal{T}(\Gamma, \mu)$ defined for every $g\in\Gamma$ by $\widehat{\Ad_{U_g}}(\rho) = \rho\circ\Ad_{U_g}$. It turns out that this induced action permutes the irreducible representations $\{ \pi_z\}_{z\in \Gamma}$ in the sense that for every $g\in \Gamma$ we have that $\pi_{z}\circ\Ad_{U_g}$ is unitarily equivalent to $\pi_{gz}$ via the unitary $U_{g^{-1},z}$. Indeed, for $n\in\N$ and $x,y \in \Gamma$ with $P^n(x,y)>0$ we have that,
\begin{equation*}
    U_{g^{-1},z}[\pi_{z}\circ\Ad_{U_g}(S^{(n)}_{x, y})] U_{g^{-1},z}^{-1} = U_{g^{-1},z} [S^{(n)}_{g^{-1}x, g^{-1}y}|_{\mathcal{H}_{z}}]U_{g^{-1},z}^{-1} = S^{(n)}_{x, y}|_{\mathcal{H}_{gz}} = \pi_{gz}(S^{(n)}_{x, y}).
\end{equation*}

We now present an extension of \cite[Proposition 3.11]{DOMA17} for finite range symmetric random walks on discrete groups.

\begin{proposition}
\label{prop:repr:uep}
Let $\Gamma$ be a countable discrete group, and $\mu$ be an admissible finitely supported lazy probability measure on $\Gamma$. Then, $\pi_x$ is a boundary representation for each $x\in\Gamma$.
\end{proposition}

\begin{proof}
Similar to \cite[Proposition 3.11]{DOMA17}, we will apply \cite[Theorem 7.2]{ARWI11} to show that each $\pi_x$ is is a boundary boundary representation by showing that it is completely strongly peaking in the sense of equation (\ref{eq:completely_peaking}). 

Using the decomposition of representations for a given irreducible representation $\pi\nsim_{ue}\pi_x$ from Section \ref{s:introduction}, we end up with two possibilities. One possibility is that $\pi$ does not annihilate the ideal of compact operators, i.e. there exists $z\in\Gamma$ such that $\pi(\K(\mathcal{H}_z)) \neq 0$. In this case $\pi$ is the unique extension of an irreducible representation on $\K(\mathcal{H}_z)$ to an irreducible representation of $\mathcal{T}(\Gamma,\mu)$, and is therefore unitarily equivalent to $\pi_z$ for some $z\neq x$. The other possibility is that $\pi$ annihilates all compact operators, which means that it annihilates the direct sum $\I(\Gamma,\mu)= \bigoplus_{z \in \Gamma} \K(\mathcal{H}_z)$. Thus, $\pi$ must factor through the quotient map $q :\mathcal{T}(\Gamma,\mu) \rightarrow \mathcal{O}(\Gamma,\mu)$ to yield an irreducible representation $\widetilde{\pi}$ of $\mathcal{O}(\Gamma,\mu)$ satisfying $\pi = \tilde{\pi}\circ q$. Hence, we have the following inequality
    \begin{equation*}
        \begin{split}
            \sup_{\pi\nsim_{ue}\pi_x}\|\pi(T)\| & \leq \max\{\sup_{z\neq x}\|\pi_z(T)\|,\, \sup_{\tilde{\pi}}\|(\tilde{\pi}\circ q)(T)\|\} \\
            & = \max\{\sup_{z\neq x}\|\pi_z(T)\|,\, \|q(T)\|\},
        \end{split}
    \end{equation*}
    where $\tilde{\pi}$ is an arbitrary irreducible representation of the Cuntz-Pimsner algebra. Hence, we have reduced the problem of showing that $\pi_x$ is completely strongly peaking to showing that there exists $T\in \mathcal{T^{+}}(\Gamma, \mu)$ satisfying the strict inequality
    \begin{equation*}
        \|\pi_x(T)\| > \max\{\sup_{z\neq x}\|\pi_z(T)\|,\, \|q(T)\|\}.
    \end{equation*}

    Since the action of $\Gamma$ on the Toeplitz algebra $\mathcal{T}(\Gamma,\mu)$ permutes the irreducible representations $\{\pi_z\}_{z\in \Gamma}$ and leaves the ideal $\I(\Gamma,\mu)$ invariant, the problem is further reduced to showing this strict inequality when $x=e$ is the identity element.

    Now, let $z\in\Gamma$ with $z\neq e$. For $n\in\N$ which we assign later, we take $T := T^{(n)}_{e, e}$ and notice that both $T^{\ast}T$ and $TT^{\ast}$ are diagonal with respect to the standard orthonormal basis. Thus, we get
    \begin{equation*}
        \left\|\pi_{e}(T)\right\|^{2} \geq \|\pi_{e}((T^{(n)}_{e, e})^{\ast}(T^{(n)}_{e, e}))(e^{(0)}_{e, e})\| = \left\|\dfrac{1}{P^{n}(e, e)}e^{(0)}_{e, e}\right\| = \dfrac{1}{P^{n}(e, e)}.
    \end{equation*}
    Since $z\neq e$, we must have that $(T_{e,e}^{(n)})^*(e_{e,z}^{(m)}) = 0$ for all $m\leq n$, so we get that 
    \begin{equation*}
        \begin{split}
            \left\|\pi_{z}(T)\right\|^{2} & = \|\pi_{z}(TT^{\ast})\| = \left\|TT^{\ast}|_{\mathcal{H}_{z}}\right\| \\
            & = \sup_{m\in \N}\left\{\left\|(T^{(n)}_{e, e})(T^{(n)}_{e, e})^{\ast}(e^{(m)}_{e, z})\right\| \,:\, n< m\wedge P^{m}( e, z) > 0\right\} \\
            & = \sup_{m\in \N} \left\{\dfrac{P^{m - n}(e, z)}{P^{m}(e, z)} \,:\, n < m\wedge (z,m) \in \ST \right\}.
        \end{split}
    \end{equation*}
    By Lemma \ref{lemma:space_time:inequality} there exists $n\in \N$ such that 
    \begin{equation*}
    \sup_{(z, m)\in\ST }K_{\ST}((e, n), (z, m)) <  \dfrac{1}{P^{n}(e, e)},
    \end{equation*}
    so that
    \begin{equation*}
        \sup_{z\neq e}\left\|\pi_{z}(T)\right\|^{2} \leq \sup_{(z, m)\in\ST }K_{\ST}((e, n), (z, m)) <
        \dfrac{1}{P^{n}(e, e)} \leq \left\|\pi_{e}(T)\right\|^{2}.
    \end{equation*}
    We are now left with showing $\|\pi_{e}(T)\| > \|q(T)\|$. Since $Q^{(n)}_{e}TT^{\ast}$ is a (finite-rank) operator in $\mathcal{I}(\Gamma,\mu)$, by definition of the quotient norm we have that
    \begin{equation*}
        \|q(T)\|^{2} \leq \|TT^{\ast}-Q^{(n)}_{e}TT^{\ast}\| = \sup_{z\in \Gamma} \sup_{m\in \N} \| [TT^{\ast}-Q^{(n)}_{e}TT^{\ast}](e_{e,z}^{(m)}) \|,
    \end{equation*}
    and since $(T_{e,e}^{(n)})(T_{e,e}^{(n)})^*(e_{e,z}^{(n)}) = \frac{\delta_{e,z}}{P^n(e,e)} \cdot e^{(n)}_{e,z}$ we therefore have that
    \begin{equation*}
        \sup_{z\in \Gamma}\sup_{m\in \N} \| [TT^{\ast}-Q^{(n)}_{e}TT^{\ast}](e_{e,z}^{(m)}) \| \leq \underset{n<m \wedge (z,m)\in \ST}{\sup}\dfrac{P^{m - n}(e, z)}{P^{m}(e, z)}.
    \end{equation*}
    Hence, we get that
    \begin{equation*}
    \|q(T)\|^2\leq \sup_{(z,m)\in \ST}K_{\ST}((e, n), (z, m)).
    \end{equation*}
    But now, again by Lemma \ref{lemma:space_time:inequality}, there exists $n\in \N$ such that 
    \begin{equation*}
    \sup_{(z, m)\in\ST }K_{\ST}((e, n), (z, m)) <  \dfrac{1}{P^{n}(e, e)},
    \end{equation*}
    so we have that 
    \begin{equation*}
        \|q(T)\|^{2} < \dfrac{1}{P^{n}(e, e)} \leq \|\pi_e(T)\|^2.
    \end{equation*}
Thus, we have established that 
\begin{equation*}
        \|\pi_e(T)\| > \max\{\sup_{z\neq e}\|\pi_z(T)\|,\, \|q(T)\|\},
\end{equation*}
so that $\pi_e$ is completely strongly peaking.
\end{proof}

Finally, we provide the main result of this section.

\begin{theorem}
\label{thm:envelope:main}

    Let $\Gamma$ be a countable discrete group, and $\mu$ be an admissible finitely supported lazy probability measure on $\Gamma$. Then, we have that 
    \begin{equation*}
        C^{\ast}_{\env}(\mathcal{T}^{+}(\Gamma, \mu))\cong\mathcal{T}(\Gamma, \mu).
    \end{equation*}
\end{theorem}

\begin{proof}
By Proposition \ref{prop:repr:uep} we know that all irreducible subrepresentations of the identity representation $\pi_x$ have the unique extension property for $z\in\Gamma$. In particular, the direct sum of all boundary representations is completely isometric on $\mathcal{T}^+(\Gamma,\mu)$. Since the Shilov ideal of $\T^+(\Gamma,\mu)$ inside $\mathcal{T}(\Gamma,\mu)$ is the intersection of all kernels of boundary representations (see \cite[Theorem 2.2.3]{ARWI69}), we see that the Shilov ideal is trivial. Since the C*-envelope of $\T^+(\Gamma,\mu)$ is the quotient of $\mathcal{T}(\Gamma,\mu)$ by the Shilov ideal, we obtain that $\mathcal{T}(\Gamma,\mu)$ is the C*-envelope of $\T^+(\Gamma,\mu)$.
\end{proof}

\bibliographystyle{alpha}
\bibliography{space-time}

\end{document}